\documentclass{birkjour}
\usepackage[utf8]{inputenc}
\usepackage{amssymb,amsmath,amsfonts,eucal,mathrsfs,amsthm} 
\usepackage[allcolors=blue]{hyperref}
\usepackage{empheq}
\usepackage{bm}
\usepackage{multicol}

 \newtheorem{theorem}{Theorem}[section]
 \newtheorem{corollary}[theorem]{Corollary}
 \newtheorem{lemma}[theorem]{Lemma}
 \newtheorem{proposition}[theorem]{Proposition}
 \theoremstyle{definition}
 
 \theoremstyle{remark}
 \newtheorem{rem}[theorem]{Remark}
 \newtheorem{example}{Example}
 \numberwithin{equation}{section}

\begin{document}

%
%
%
%
%
%
%
%
%

\title[On a class of hypersurfaces of a product of two space forms]
 {On a class of hypersurfaces of a product of two space forms}

\thanks{Corresponding author: arnandonelio@usp.br. The first author was supported by  CAPES-PROEX grant 12498478/D. The second author was partially supported by  Fapesp grant 2022/16097-2 and CNPq grant 307016/2021-8.}

\author[Arnando Carvalho and Ruy Tojeiro]{A. N. S. Carvalho 
\and R. Tojeiro}

\address{%
Institute of Mathematics and Computer Sciences (ICMC) \\
University of São Paulo (USP)\\
SP 13566590  São Carlos\\
Brazil}

\email{arnandonelio@usp.br and tojeiro@icmc.usp.br}

\subjclass{53 B25}

\keywords{Hypersurfaces in class $\mathcal{A}$, hypersurfaces with constant mean curvature, hypersurfaces with constant product angle function, products of space forms. }

\date{\today}

\begin{abstract}
We extend the notion of \emph{class $\mathcal{A}$}, previously defined for hypersurfaces of $\mathbb{S}^n\times \mathbb{R}$ and $\mathbb{H}^n\times \mathbb{R}$, where  $\mathbb{S}^n$ and $\mathbb{H}^n$ denote the sphere and hyperbolic space, respectively, to hypersurfaces of a product $\mathbb{Q}_{c_1}^{k} \times \mathbb{Q}_{c_2}^{n-k+1}$, $2\leq k\leq n-1$, of two space forms with arbitrary constant curvatures $c_1$ and $c_2$. When $c_1$ and $c_2$ are not both equal to $0$, 
such a hypersurface is in class $\mathcal{A}$ if and only if it has a flat normal bundle when regarded as a submanifold of the underlying flat ambient space. We provide an explicit local construction of all hypersurfaces in class $\mathcal{A}$ that do not split in any open subset in terms of parallel families of hypersurfaces of the factors, and show how it simplifies for the hypersurfaces within this class that have a constant product angle function.  We also show that hypersurfaces with constant mean curvature in class $\mathcal{A}$  are given in terms of parallel families of isoparametric hypersurfaces in each factor and a solution of a second order ODE. Finally, we classify hypersurfaces with constant mean curvature in class~$\mathcal{A}$ that have a constant product angle function. 
\end{abstract}

\maketitle
\section{Introduction}
Some relevant classes of hypersurfaces of the product spaces $\mathbb{S}^n\times \mathbb{R}$ and $\mathbb{H}^n\times \mathbb{R}$, where $\mathbb{S}^n$ and $\mathbb{H}^n$ denote the $n$-dimensional sphere and hyperbolic space, respectively, such as hypersurfaces with constant sectional curvature, umbilical hypersurfaces, and \emph{hypersurfaces with constant angle}, that is, hypersurfaces whose normal vector field make a constant angle with the unit vector field $\partial /\partial t$ that spans the factor $\mathbb{R}$, share the property that the tangent component of $\partial /\partial t$ is a principal direction. These
hypersurfaces were studied in \cite{To} and named \emph{hypersurfaces in class} $\mathcal{A}$ in \cite{MT2}, where this notion was extended to submanifolds of higher codimension. In \cite{To}, it was shown that 
a hypersurface of $\mathbb{S}^n\times \mathbb{R}$ (respectively,  $\mathbb{H}^n\times \mathbb{R}$) is in class $\mathcal{A}$ if and only if
  it has a flat normal bundle when regarded as a submanifold with codimension two of Euclidean space $\mathbb{R}^{n+2}$ (respectively, Lorentzian space $\mathbb{L}^{n+2}$). 

 In this article,  we extend the notion of class $\mathcal{A}$ to hypersurfaces of a product $\mathbb{Q}_{c_1}^{k} \times \mathbb{Q}_{c_2}^{n-k+1}$, $2\leq k\leq n-1$,  of two space forms with arbitrary constant curvatures $c_1$ and $c_2$ (see the definition in the beginning of Section~\ref{classA}). When $c_1$ and $c_2$ are not both equal to $0$, it turns out that a hypersurface of  $\mathbb{Q}_{c_1}^{k} \times \mathbb{Q}_{c_2}^{n-k+1}$ is in class $\mathcal{A}$ if and only if it has a flat normal bundle when regarded as a submanifold of the underlying flat ambient space, thereby showing that our definition is a natural extension of the one for hypersurfaces of  
 $\mathbb{S}^n\times \mathbb{R}$ and $\mathbb{H}^n\times \mathbb{R}$.
 
 Our interest in studying hypersurfaces of $\mathbb{Q}_{c_1}^{k} \times \mathbb{Q}_{c_2}^{n-k+1}$, $2\leq k\leq n-1$, in class $\mathcal{A}$ is twofold. First, as in the case where the ambient space is $\mathbb{S}^n\times \mathbb{R}$ and $\mathbb{H}^n\times \mathbb{R}$, the class $\mathcal{A}$ includes some relevant classes of hypersurfaces, such as hypersurfaces with constant sectional curvature of dimension $n\geq 4$, or of dimension $n=3$ and shape operator of maximal rank (see \cite{KNT}), umbilical hypersurfaces and doubly rotational hypersurfaces (see Example \ref{doubly}). Second,  hypersurfaces in class $\mathcal{A}$ provide interesting examples of hypersurfaces in other important classes, such as those with constant mean curvature and those with constant product angle function, which are investigated in Sections~$4$ and $5$. In the former case, an interesting connection with isoparametric hypersurfaces of space forms arises.
 
Our main result is an explicit local description of all hypersurfaces of $ \mathbb{Q}_{c_1}^{k} \times \mathbb{Q}_{c_2}^{n-k+1}$, $2\leq k\leq n-1$, that are in class $\mathcal{A}$ and do not split in any open subset, in terms of parallel families of hypersurfaces of the factors (see Theorem~\ref{classAcylinders}). 
Then we show how such a description simplifies for the hypersurfaces within this class that have a constant product angle function. The
\emph{product angle function} of a hypersurface $f\colon M^n\to \mathbb{Q}_{c_1}^{k} \times \mathbb{Q}_{c_2}^{n-k+1}$ is defined by $\Theta=\langle PN, N\rangle$, where $N$ is a unit normal vector field along $f$ and $P$ is the product structure of $\mathbb{Q}_{c_1}^{k} \times \mathbb{Q}_{c_2}^{n-k+1}$, given by $P(X_1, X_2)=(X_1, -X_2)$
for tangent vector fields $X_1$ and $X_2$. It has
proven to contain important information about the geometry of the hypersurface.
   
We also show that hypersurfaces of $\mathbb{Q}_{c_1}^{k} \times \mathbb{Q}_{c_2}^{n-k+1}$, $2\leq k\leq n-1$, in class~$\mathcal{A}$ that have constant mean curvature and do not split in any open subset are given, in terms of the general description of hypersurfaces in class $\mathcal{A}$  in Theorem \ref{classAcylinders},  by means of parallel families of isoparametric hypersurfaces of the factors and a solution of a certain second order ODE (see Theorem~\ref{thm:cmc}). 
   
The hypersurfaces of $\mathbb{S}^2\times \mathbb{S}^2$ and $\mathbb{H}^2\times \mathbb{H}^2$ with constant mean curvature and a constant product angle function were classified in \cite{LVWY} and \cite{LVWY2}, respectively, after a partial classification for $\mathbb{S}^2\times \mathbb{S}^2$  was obtained in \cite{U}. Our last results provide a classification of all hypersurfaces of $\mathbb{Q}_{c_1}^{k} \times \mathbb{Q}_{c_2}^{n-k+1}$, $2\leq k\leq n-1$, in class $\mathcal{A}$ that have constant mean curvature and a constant product angle function  (see Theorem \ref{prodspheres} and  Propositions \ref{thm:cmcpacgeral} to \ref{eucb}). In particular, several new minimal hypersurfaces of certain products of two spheres are constructed.

\section{Preliminaries}

In this section, following  \cite{LTV}, we summarize the basic equations of hypersurfaces of a product of two space forms. 

Given an isometric immersion  
$
f\colon M^n \to \mathbb{Q}_{c_1}^{k}\times \mathbb{Q}_{c_2}^{n-k+1}
$, $2\leq k\leq n-1$, 
into a product of two space forms of constant sectional curvatures $c_1$ and $c_2$, 
let  $N$ be a (local) unit normal vector field along $f$ and let $A$ be the shape operator of $f$ with respect to $N$. We denote by  
$
\pi_2\colon \mathbb{Q}_{c_1}^{k}\times \mathbb{Q}_{c_2}^{n-k+1}\to \mathbb{Q}_{c_2}^{n-k+1}
$ 
both the projection onto 
$\mathbb{Q}_{c_2}^{n-k+1}$ 
and its derivative, and define $t=t_f\in C^{\infty}(M)$,
$\xi=\xi_f\in \mathfrak{X}(M)$ and $R=R_f\in \Gamma(T^*M\otimes TM)$ by
\begin{equation} \label{pi2N}\pi_2N=f_*\xi+tN
\end{equation}
and 
\begin{equation}\label{eq:L} \pi_2 f_*X=f_*RX+\langle X, \xi\rangle N,
\end{equation}
which imply that
\begin{equation} \label{pi1N}
\pi_1N=-f_*\xi+(1-t)N
\end{equation}
and 
\[ \pi_1 f_*X=f_*(I-R)X-\langle X, \xi\rangle N.
\]
From
$\pi_2^2=\pi_2$
it follows that
\begin{equation}\label{eq:pi1}
R(I-R)X=\langle X, \xi\rangle \xi,\,\,\,\,\,((1-t)I-R)\xi=0\,\,\,\,\,\mbox{and}\,\,\,\,\, t(1-t)=\|\xi\|^2
\end{equation}
for all $X\in \mathfrak{X}(M)$. Then \eqref{eq:L} implies that $R$ is symmetric and, by the first equation in \eqref{eq:pi1}, at each  $x\in M^n$ either $\xi$ vanishes at $x$ and $T_xM$ decomposes orthogonally as $T_xM=\ker R \oplus \ker (I-R)$, or $T_xM=\ker R \oplus \ker (I-R)\oplus \mbox{span} \, \{\xi\}$.  In the latter case, in view of the second equation in \eqref{eq:pi1},  the eigenvalues of $R$ are $0$, $1$ and $r=1-t$, with corresponding eigenspaces $\ker R$, $\ker (I-R)$ and $ \mbox{span} \, \{\xi\}$. In Lemma 3.2 of~\cite{MT}, it was shown that the subspaces $\ker R$ and $\ker (I-R)$ give rise to smooth subbundles of $TM$  on open subsets where they have constant dimension.

A hypersurface 
$
f\colon M^{n}\rightarrow \mathbb{Q}_{c_1}^{k}\times \mathbb{Q}_{c_2}^{n-k+1}
$, $2\leq k\leq n-1$,
is said to \emph{split} if one of the following holds:
\begin{itemize}
    \item  $M^n$ is (isometric to) a Riemannian product $N_1^{k-1}\times \mathbb{Q}_{c_2}^{n-k+1}$ and
    \[
    f=f_{1}\times i_{2}\colon N_1^{k-1}\times \mathbb{Q}_{c_2}^{n-k+1}\to  \mathbb{Q}_{c_1}^{k}\times \mathbb{Q}_{c_2}^{n-k+1};
    \]
    \item 
     $M^n$ is (isometric to) a Riemannian product $\mathbb{Q}_{c_1}^{k}\times N_2^{n-k}$ and 
     \[f=i_{1}\times f_{2}\colon \mathbb{Q}_{c_1}^{k}\times N_2^{n-k}\rightarrow \mathbb{Q}_{c_1}^{k}\times \mathbb{Q}_{c_2}^{n-k+1},\]
\end{itemize}
where 
$f_{1}\colon N_1^{k-1}\rightarrow \mathbb{Q}_{c_1}^{k}$ and $f_{2}\colon N_2^{n-k}\rightarrow  \mathbb{Q}_{c_2}^{n-k+1}$ 
are hypersurfaces and $i_{1}$ and $i_{2}$ are the respective identity maps. 
We say that the hypersurface $f$  \emph{splits locally} if, for each $x\in M^n$, there exists an open neighborhood isometric to a Riemannian product, restricted to which $f$ splits.

\begin{proposition}\label{product locally splits}
     Let $f\colon M^{n}\rightarrow \mathbb{Q}_{c_1}^{k}\times \mathbb{Q}_{c_2}^{n-k+1}$, $2\leq k\leq n-1$, be a hypersurface. Then the following assertions are equivalent:
\begin{itemize}
\item[(a)] $f$ splits locally;
\item[(b)] The tangent vector field $\xi$ vanishes everywhere.
\item[(c)] The function $r$ has the constant value $0$ or $1$;
\item[(d)] The product angle function $\Theta$ has the constant value $1$ or $-1$.
\end{itemize}
\end{proposition}
\begin{proof}
The equivalence between (a) and (b) is a particular case of  Proposition~$3.3$ in \cite{MT}. By the third equation in \eqref{eq:pi1},  $\xi$ vanishes everywhere if and only if $t=0$ or $t=1$ everywhere, or equivalently, if and only if $r=0$ or $r=1$ everywhere. Thus, (b) and (c) are equivalent. Finally, \eqref{pi2N} and~\eqref{pi1N} give
\[
\Theta=\langle \pi_1N-\pi_2N, N\rangle=1-2t=2r-1,
\]
and the equivalence between $(c)$ and $(d)$ follows. 
\end{proof}

Computing the tangent and normal components in
$\nabla\pi_2=\pi_2\nabla$, 
applied to both tangent and normal vectors, and using the Gauss and Weingarten formulae,  we obtain
\begin{equation}\label{derR}
(\nabla_XR)Y= \langle Y, \xi\rangle AX+\langle AX,Y\rangle \xi,
\end{equation}
\begin{equation}\label{derS2}
\nabla_X\xi=(tI-R)AX
\end{equation}
for any $X\in\mathfrak{X}(M)$, 
and 
\begin{equation}\label{dert2}
\operatorname{grad} t=-2A\xi.
\end{equation}

We will need the following fact.

\begin{lemma}\label{int} 
The distributions $\ker R$ and $\ker (I-R)$ are integrable.
\end{lemma}
\begin{proof}
If $X, Y\in \ker R$, then \eqref{derR} gives
$$R\nabla_XY=-\langle AX,Y\rangle \xi=-\langle AY,X\rangle \xi=R\nabla_YX,$$
hence $R[X,Y]=0$. Similarly, if $X, Y\in \ker (I-R)$, then
$$(I-R)\nabla_XY=\langle AX,Y\rangle \xi=\langle AY,X\rangle \xi=(I-R)\nabla_YX,$$
thus  $(I-R)[X,Y]=0$. 
\end{proof}
Under the assumption that $AR=RA$, we also have the following (see Lemma~$8$ of \cite{CT}). 
\begin{lemma}\label{totgeo} 
If $AR=RA$, then the distributions $(\ker R)^\perp$ and $(\ker (I-R))^\perp$ are totally geodesic.
\end{lemma}

 The Gauss and Codazzi equations of $f$ are
\begin{equation*}\label{Gauss}
\mathcal{R}(X,Y)=c_1(X\wedge Y-X\wedge RY-RX\wedge Y) + (c_1+ c_2)RX\wedge RY +AX\wedge AY
\end{equation*}
and 
\begin{equation*}\label{codazzi}
(\nabla_XA)Y-(\nabla_YA)X=
(c_1(I-R)-c_2R)(X\wedge Y)\xi
\end{equation*}
for all $X,Y\in\mathfrak{X}(M),$ where  $\mathcal{R}$ denotes the curvature tensor  of $M.$

Given a hypersurface 
$
f\colon\,M^n
\to \mathbb{Q}_{c_1}^{k}\times \mathbb{Q}_{c_2}^{n-k+1}
$, 
denote 
$\tilde f=j\circ f$, 
where 
\begin{equation}\label{eq: inclusion j}
j\colon \mathbb{Q}_{c_1}^{k}\times\mathbb{Q}_{c_2}^{n-k+1}\to \mathbb{R}^{N_1}_{\sigma(c_1)}\times \mathbb{R}^{N_2}_{\sigma(c_2)}=\mathbb{R}^{N_1+N_2}_{\mu}
\end{equation} 
is the inclusion. 
Here, for $c\in \mathbb{R}$,
$$
\sigma(c)=
\begin{cases}
0, & \text{if } c\geq 0,\\
1, & \text{if }  c<0,
\end{cases}
\qquad \mu=\sigma(c_1)+ \sigma(c_2),
$$
$$
N_1=
\begin{cases}
k+1, & \text{if }  c_1\neq 0,\\
k, & \text{if } c_1=0,\\
\end{cases}
\qquad\text{and}\qquad
N_2=
\begin{cases}
n-k+2, & \text{if } c_2\neq 0,\\
n-k+1, & \text{if }  c_2=0.\\
\end{cases}
$$
If $c_i\neq 0$, $1\leq i\leq 2$, write $c_i=\epsilon_i/r_i^2$, where $\epsilon_i$ is either $1$ or $-1$, according to whether $c_i>0$ or $c_i<0$, respectively. Let 
$
\tilde \pi _i  \colon \mathbb{R}^{N_1+N_2}_{\mu} \to  \mathbb{R}^{N_i}_{\sigma(c_i)}
$ 
denote the orthogonal projection, $1\leq i\leq 2$. 
Then the vector fields  
$$
\nu_1=\nu_1^{\tilde f}=\frac{1}{r_1}\tilde\pi_1\circ \tilde f
\quad \mbox{and} \quad
\nu_2=\nu_2^{\tilde{f}}=\frac{1}{r_2}\tilde\pi_2\circ \tilde f
$$
are  normal  to $\tilde f$, $\langle\nu_i, \nu_i\rangle=\epsilon_i$, $1\leq i\leq 2$, 
\begin{equation}
\label{eq:nu}A^{\tilde f}_{\nu_1} = -\frac{1}{r_1}(I-R)\,\,\,\,\,\mbox{and}\,\,\,\,\,\,A^{\tilde f}_{\nu_2} = -\frac{1}{r_2}R.
\end{equation}
 If, say,  $c_1=0$ and $0\neq c_2=\epsilon_2/r_2^2$, then the vector field
$
\nu_2=\frac{1}{r_2}\tilde\pi_2\circ \tilde f
$
is normal  to $\tilde f$, $\langle\nu_2, \nu_2\rangle=\epsilon_2$, and 
\begin{equation}
\label{eq:nu2}A^{\tilde f}_{\nu_2} = -\frac{1}{r_2}R.
\end{equation}

\section{Hypersurfaces in class $\mathcal{A}$} \label{classA}

A hypersurface of $\mathbb{S}^n\times \mathbb{R}$ or  $\mathbb{H}^n\times\mathbb{R}$ is said to be in class $\mathcal{A}$ if the tangent component of the vector field $\partial/\partial t$ spanning the factor $\mathbb{R}$ is a principal direction. This is easily seen to be equivalent to the hypersurface having a flat normal bundle 
when regarded as a submanifold with codimension two of $\mathbb{R}^{n+2}$ or  $\mathbb{L}^{n+2}$.

We say that a hypersurface $f\colon M^{n}\rightarrow \mathbb{Q}_{c_1}^{k}\times \mathbb{Q}_{c_2}^{n-k+1}$, $2\leq k\leq n-1$,   
is in \emph{class}~$\mathcal{A}$ if its shape operator $A$ commutes with the tensor $R$. If $c_1$ and $c_2$ are not both equal to $0$, then \eqref{eq:nu}, \eqref{eq:nu2} and the Ricci equation for $\tilde f=j\circ f$, where $j$ is the inclusion defined in \eqref{eq: inclusion j}, imply that  $f$ is in class~$\mathcal{A}$ if and only if $\tilde f$ has a flat normal bundle. In particular, this extends the notion of class~$\mathcal{A}$ for hypersurfaces of $\mathbb{S}^n\times \mathbb{R}$ or  $\mathbb{H}^n\times\mathbb{R}$.

Clearly, any hypersurface $f\colon M^{n}\rightarrow \mathbb{Q}_{c_1}^{k}\times \mathbb{Q}_{c_2}^{n-k+1}$, $2\leq k\leq n-1$, 
that splits is in class~$\mathcal{A}$. Next, we show how to construct hypersurfaces in class $\mathcal{A}$ that do not split in any open subset.

Let $h\colon M_1^{k-1} \to \mathbb{Q}_{c_1}^{k}$ and $g \colon  M_2^{n-k} \to  \mathbb{Q}_{c_2}^{n-k+1}$, $2 \leq  k \leq n -1$, be  hypersurfaces with (local) unit  normal vector fields $N^h$ and $N^g$, respectively. 
Consider the families of parallel hypersurfaces 
$
h_s\colon M_1^{k-1}\to \mathbb{Q}_{c_1}^{k}
$ 
and 
$
g_s \colon  M_2^{n-k} \to \mathbb{Q}_{c_2}^{n-k+1}
$ 
of $h$ and $g$, which are given, for $c_1\neq 0\neq c_2$, as maps into  $\mathbb{R}^{N_1}_{\sigma(c_1)}$ and 
$ \mathbb{R}^{N_2}_{\sigma(c_2)}$, respectively, by
\begin{equation*}\label{h_s}
h_s (x_1) = C_{\epsilon_1}( s / r_1) h(x_1) + r_1 S_{\epsilon_1}(s/r_1) N^h(x_1), \,\, \epsilon_1 = \operatorname{sgn}(c_1), \,\, c_1=\epsilon_1/r_1^2,
\end{equation*}
and 
\begin{equation*}\label{g_s}
g_s (x_2) = C_{\epsilon_2}( s / r_2) g(x_2) + r_2 S_{\epsilon_2}(s/r_2) N^g(x_2), \,\, \epsilon_2 = \operatorname{sgn}(c_2), \,\, c_2=\epsilon_2/r_2^2,
\end{equation*}
for all  $x_1 \in M_1^{k-1}$ and $x_2 \in M_2^{n-k}$. Here, for $\epsilon \in \{-1, 1\}$, 
$$
C_\epsilon(s)=
\begin{cases}
\cos  s, & \text{if } \epsilon =1,\\
\cosh  s, & \text{if } \epsilon = -1,
\end{cases}
\qquad\text{and}\qquad
S_\epsilon(s)=
\begin{cases}
\sin  s, & \text{if } \epsilon=1,\\
\sinh s, & \text{if } \epsilon=-1.
\end{cases}
$$
If $c_1=0$, then 
\begin{equation*}\label{h_s}
h_s (x_1) = h(x_1) + s N^h(x_1),
\end{equation*}
and similarly if $c_2=0$. 

Let $a,b \colon I= (-\delta,\delta) \to \mathbb{R}$ be smooth functions such that 
\begin{equation}\label{ab}
a(0)=0=b(0),\,\,\,  a'(s), b'(s) >0 \,\text{ and } \, (a'(s))^2 + (b'(s))^2 =1
\end{equation}
for all $s \in I.$
Define 
$$
f \colon M^n=I \times M_1^{k-1} \times M_2^{n-k} \to \mathbb{Q}_{c_1}^{k} \times \mathbb{Q}_{c_2}^{n-k+1} \subset \mathbb{R}^{N_1}_{\sigma(c_1)}\times \mathbb{R}^{N_2}_{\sigma(c_2)}=\mathbb{R}^{N_1+N_2}_{\mu}$$
 by 
\begin{equation}\label{definition of f-1}
f(s,x_1,x_2) = h_{a(s)}(x_1) +  g_{b (s)}(x_2).
\end{equation}

Next, we show that the restriction of $f$ to the subset of $M^n$ of its regular points defines a hypersurface in class $\mathcal{A} $ that does not split in any open subset.

For each $x=(s,x_1,x_2) \in M$, we have 
\begin{equation}\label{eq1 - derivada da f}
 f_*(x)X_1 = (h_{a(s)})_\ast(x_1) X_1\quad 
f_*(x)X_2 = (g_{b(s)})_\ast(x_2) X_2
\end{equation}
for all $X_1 \in T_{x_1}M_1$ and $X_2 \in T_{x_2}M_2$, and 
\begin{equation}\label{eq2 - derivada da f} 
f_*(x) \dfrac{\partial}{\partial s} =  a'(s) N_{a(s)}^h(x_1) + b'(s)N^g_{b(s)}(x_2).
\end{equation}
Here $\partial / \partial s$ is a unit vector field along $I$ and, for $c_1\neq 0\neq c_2$,
$$
N^h_{a(s)}(x_1) = -\frac{\epsilon_1}{r_1}S_{\epsilon_1}\left( \frac{a(s)}{r_1} \right)h(x_1) + C_{\epsilon_1}\left( \frac{a(s)}{r_1} \right)N^h(x_1)
$$ 
and 
$$
N^g_{b(s)}(x_2) =  -\frac{\epsilon_2}{r_2}S_{\epsilon_2}\left( \frac{b(s)}{r_2} \right) g(x_2)+ C_{\epsilon_2}\left( \frac{b(s)}{r_2} \right) N^g(x_2).
$$
If $c_1= 0$, then $N^h_{a(s)}(x_1)=N^h(x_1)$ for all $s\in I$ and $x_1\in M_1^{k-1}$. Similarly, if $c_2=0$.
Therefore, a point $x=(s,x_1,x_2) \in M^n$ is regular for $f$ if and only if $h_{a (s)}$ is regular at $x_1$ and $g_{b (s)}$ is regular at  $x_2$,  in which case $N_{a(s)}^h(x_1)$ is a unit normal vector to $h_{a (s)}$ at $x_1$ and  $N_{b (s)}^g$ is a unit normal vector to $g_{b(s)}$ at $x_2.$

It follows from \eqref{eq1 - derivada da f} and \eqref{eq2 - derivada da f} that the vector field given by
\begin{equation*}
\eta(x) = - b '(s)N_{a(s)}^h(x_1) + a'(s)N^g_{b(s)} (x_2)
\end{equation*}
is a unit normal vector to $f$ at $(s,x_1,x_2)$. Differentiating $\eta$ with respect to the connection $\tilde \nabla$ in $\mathbb{R}_\mu^{N_1+N_2}$, we obtain
\begin{equation}\label{derivada1 --->eta}
\tilde \nabla_{X_1} \eta = b '(s)\, (h_a)_\ast A_a^h X_1 \quad \text{and}\quad
\tilde \nabla_{X_2} \eta = -a'(s)\, (g_{b})_\ast A_b^g X_2, 
\end{equation}
for all $X_1 \in T_{x_1}M_1$ and $X_2 \in T_{x_2}M_2$, where $A_{a}^h$ and $A_{b}^g$ denote the shape operators of $h_{a (s)}$ and $g_{ b(s)}$ with respect to the unit normal vector fields $N^h_a $ and $N^g_b$, respectively,
and
\begin{equation}\label{derivada2 --->eta}
\tilde \nabla_{\frac{\partial}{\partial s}} \eta = - b ''(s)\, N_a^h(x_1) + a''(s)\, N^g_b(x_2) +a'(s)b'(s)(c_1h_{a(s)} -c_2g_{b(s)}).
\end{equation}
 Thus,
\begin{equation*}\label{ddsprincipal}
\langle f_* \frac{\partial}{\partial s}, \tilde \nabla_{X_1} \eta \rangle = 0 =\langle f_* \frac{\partial}{\partial s}, \tilde \nabla_{X_2} \eta\rangle
\end{equation*}
for all $X_1 \in TM_1$ and   $X_2 \in T_{x_2}M_2$, and 
\begin{equation*}\label{x1x2}
\langle f_* X_2, \tilde \nabla_{X_1} \eta \rangle = 0 =\langle f_* X_1, \tilde \nabla_{X_2} \eta \rangle.
\end{equation*}
Writing $a'(s) = \cos \theta(s)$ and $b'(s) = \sin \theta (s) $ for $\theta \in C^\infty(I)$,  from \eqref{eq2 - derivada da f}  and \eqref{derivada2 --->eta} we obtain
\[
\langle \tilde \nabla_{\frac{\partial}{\partial s}} \eta, f_\ast \frac{\partial}{\partial s}\rangle 
= -b''(s)a'(s) + a''(s)b'(s)  =-\theta'(s) 
\]
for all $s \in I$. Thus, 
$
A^f_\eta \frac{\partial}{\partial s} = \theta'(s) \frac{\partial}{\partial s} .
$
Now, by the first equation in \eqref{derivada1 --->eta}, for any $x=(s,x_1,x_2)$ we have
\begin{equation*}
\begin{split}
    f_* A^f_\eta X_1 &= -\tilde \nabla_{X_1} \eta \\
    &= - b '(s) (h_a)_\ast A_a^h X_1 \\
    &=-b '(s) f_\ast A^h_a X_1,
\end{split}
\end{equation*}
for all $X_1 \in T_{x_1}M_1$. Similarly, the second equation in \eqref{derivada1 --->eta} gives
\begin{equation*}
    f_* A^f_\eta X_2 = a'(s) f_\ast A_b^g X_2,
\end{equation*}
for any $X_2 \in T_{_2}M_2.$
Therefore
$$
A^f_\eta X_1 = - b'(s) A_a^h X_1\quad \text{and}\quad
A^f_\eta X_2 = a'(s)\, A_b^g X_2,
$$
for all $X_1 \in T_{x_1}M_1$ and $X_2 \in T_{x_2}M_2$. 
Summarizing, the shape operator $A^f_\eta$ at $x=(s,x_1,x_2)$ is given by
\begin{equation}\label{A-eigenspaces}
A^f_\eta|_{T_{x_1}M_1}=-b'(s)\, A_a^h,\quad 
{A^f_\eta}|_{T_{x_2}M_2} =  a'(s)  A_b^g, \quad \text{and} \quad
A^f_\eta \frac{\partial}{\partial s} = \theta'(s) \frac{\partial}{\partial s}.
\end{equation}

Now, a straightforward computation gives
\begin{equation}\label{RR}
R_f|_{TM_1}=0,\,\,\,\,\,\, R_f|_{TM_2}=I,\,\,\,\,\,\,R_f\dfrac{\partial}{\partial s}=\sin^2 \theta(s)\dfrac{\partial}{\partial s},
\end{equation}
and
\begin{equation}\label{R at dt}
 \xi=\frac{1}{2}\sin 2\theta(s) \,\dfrac{\partial}{\partial s}.
\end{equation}

It follows from \eqref{A-eigenspaces} and \eqref{RR} that $A^f_\eta$ and $R_f$ commute, hence $f$ belongs to class~$\mathcal{A}$. Moreover, by Proposition \ref{product locally splits}, it does not split in any open subset, for $\sin 2\theta(s) \neq 0$ for all $s\in I$. 
    
We have thus proven the direct statement of the following result.

\begin{theorem} \label{classAcylinders} The map $f$ defines, on the open subset of $I\times M_1\times M_2$ of its regular points, a hypersurface of  $\mathbb{Q}_{c_1}^{k} \times \mathbb{Q}_{c_2}^{n-k+1}$, $2\leq k\leq n-1$, in class $\mathcal{A}$. 

Conversely, any hypersurface $f\colon M^{n}\to  \mathbb{Q}_{c_1}^{k} \times \mathbb{Q}_{c_2}^{n-k+1}$, $2\leq k\leq n-1$,
in class $\mathcal{A}$ that does not split in any open subset is given locally in this way in an open and dense subset of $M^n$.
\end{theorem}

To prove the converse statement of Theorem \ref{classAcylinders}, we first recall some terminology, following \cite{rs}.

A \emph {net}\index{Net} $\mathcal{E}=(E_i)_{i=1,\ldots,r}$ on a differentiable
manifold $M$ is a splitting $TM=\oplus_{i=1}^r E_i$ of its tangent bundle 
by a family of integrable subbundles.  If $M$ is a Riemannian manifold and 
the subbundles of $\mathcal{E}$ are mutually orthogonal, then $\mathcal{E}$ is called
an \emph{orthogonal net}.

For a product manifold $M=\Pi_{i=1}^r M_i$ let $\pi_i\colon M\to M_i$ 
denote the projection onto $M_i$. 
The map $\tau^{\bar x}_i\colon M_i\to M$, for 
$\bar x=(\bar x_1,\ldots,\bar x_r)\in M$, stands for the  inclusion 
of $M_i$ into $M$ given by
$$
\tau^{\bar x}_i(x_i)=(\bar x_1,\ldots,x_i,\ldots,\bar x_r),\;\;1\leq i\leq r.
$$
The \emph{product net}\index{Product net} of 
$M=\Pi_{i=1}^r M_i$  is the net 
$\mathcal{E}=(E_i)_{i=1,\ldots,r}$ on $M$ defined by
$$
E_i(x)={\tau^{x}_i}_*T_{x_i}M_i,\;\;1\leq i\leq r,
$$
for any $x=(x_1,\ldots,x_r)\in M$. If  $N$ is a smooth manifold endowed with a net $\mathcal{F}=(F_i)_{i=1,\ldots, r}$, 
a \emph{product representation}\index{Product representation} of $\mathcal{F}$ is 
a smooth diffeomorphism 
$\psi\colon M\to N$ of a product manifold $M=\Pi_{i=1}^r M_i$ onto $N$ such that
$\psi_*E_i(x)= F_i(\psi(x))$ for all $x\in M$, $1\leq i\leq r$.

A net $\mathcal{E}=(E_i)_{i=1,\ldots,r}$ on a 
smooth manifold  $M$ is  said to be \emph{locally decomposable} if for every point $x\in M$ there exist a
neighborhood $U$ of $x$ and a product representation $\psi \colon \Pi_{i=1}^r M_i\to (U, \mathcal{E}|_U)$.
It was shown in Theorem~$1$ of \cite{rs} that the net $\mathcal{E}=(E_i)_{i=1,\ldots,r}$   is locally decomposable if and only if $E_i^{\perp}:=\oplus_{j\neq i}
E_i$ is integrable for $i=1,\ldots, r$.

\begin{proof}[Proof of the converse statement of Theorem \ref{classAcylinders}]
 Since $f$ does not split in any open subset, by Proposition \ref{product locally splits} there exists an open and dense subset $\mathcal{U}\subset M^n$ where the vector field $\xi$ does not vanish. We will show that the statement holds on $\mathcal{U}$ and will always argue for $f|_{\mathcal{U}}$. 
 
 Since $\ker R$ and $\ker (I-R)$  are integrable distributions by Lemma \ref{int}, while  $(\ker R)^\perp$ and  $(\ker (I-R))^\perp$  are totally geodesic by Lemma \ref{totgeo}, it follows from Theorem $2.7$ in \cite{To2} that there exists locally a product representation  $\Phi\colon I\times M_1^{k-1}\times M_2^{n-k}\to M^{n}$ of the net $\mathcal{E}=((\operatorname{span} \{\xi\}, \ker R, \ker (I-R))$ which is an isometry with respect to a polar metric $\pi_0^*ds^2+\pi_1^*(g_1\circ \pi_0) +\pi_2^*(g_2\circ \pi_0)$ on $I\times M_1^{k-1}\times M_2^{n-k}$, where $s\in I\mapsto g_1(s)$ and $s\in I\mapsto g_2(s)$ are smooth one-parameter families of metrics on $M_1$ and $M_2$, respectively, indexed on the open interval $I\subset \mathbb{R}$, which we can assume contains $0$.
 
 For each $x_1^0\in M_1^{k-1}$ and $x_2^0\in M_2^{n-k}$, the curve $\gamma=\gamma_{(x_1^0, x_2^0)}\colon I\to M^n$  given by $\gamma(s)=\Phi(s, x_1^0, x_2^0)$
is an integral curve of the unit vector field $\hat\xi=\xi/\|\xi\|$. 
Denote $\tilde f=f\circ \Phi$.  For all $x=(s, x_1, x_2)\in I\times M_1^{k-1}\times M_2^{n-k}$ and all $X\in T_{x_1}M_1$, by \eqref{eq:L} we have 
$${\pi_2}_*\tilde f_*X={\pi_2}_*f_*\Phi_*X=f_*R\Phi_*X+\langle\Phi_*X, \xi\rangle N=0,$$
for $\Phi_*X\in \ker R$. Similarly,
${\pi_1}_*\tilde f_*X=0$ for all $X\in T_{x_2}M_2$.
Therefore, the map $\pi_1\circ \tilde f$ does not depend on $x_2$, and $\pi_2\circ \tilde f$ does not depend on $x_1$.

We claim that the curves
$\alpha\colon I\to \mathbb{Q}_{c_2}^{n-k+1}$ and $\beta\colon I\to \mathbb{Q}_{c_1}^{k}$, given by 
$\alpha=\pi_2\circ  f\circ \gamma$ and 
$\beta=\pi_1\circ  f\circ \gamma$,  are pre-geodesics of 
$\mathbb{Q}_{c_2}^{n-k+1}$ and $\mathbb{Q}_{c_1}^{k}$, respectively. 
We argue for $\alpha$, the argument for $\beta$ being similar. 
By \eqref{eq:L} and the last equation in \eqref{eq:pi1}, we have
$$
\alpha'= {\pi_2}f_*\hat\xi\\
= f_*R\hat\xi+\langle \hat\xi, \xi\rangle N\\
= rf_*\hat\xi+r^{1/2}(1-r)^{1/2}N,  
$$
whose length is $r^{1/2}$. Therefore, to prove our claim, we need to show that $\bar \nabla_{\xi} {\pi_2}f_*(r^{-1/2}\hat\xi)=0$, where $\bar \nabla$ is the Levi-Civita connection of 
$\mathbb{Q}_{c_1}^{k} \times \mathbb{Q}_{c_2}^{n-k+1}$. We have
\begin{equation} \begin{array}{l}\label{barnabla} \bar{\nabla}_\xi {\pi_2}f_*(r^{-1/2}\hat\xi)=\bar{\nabla}_\xi r^{-1/2} {\pi_2}f_*\hat \xi
 \vspace{1ex}\\
    \hspace*{16.5ex}= \bar{\nabla}_\xi r^{-1/2}(rf_*\hat \xi+\langle \xi, \hat \xi\rangle N)\vspace{1ex}\\
    \hspace*{16.5ex}= \bar{\nabla}_\xi r^{1/2}f_*\hat \xi+\bar{\nabla}_\xi (1-r)^{1/2}N \vspace{1ex}\\
    \hspace*{16.5ex}= \xi( r^{1/2})f_*\hat \xi+ r^{1/2} \bar{\nabla}_\xi f_*\hat \xi+ \xi ((1-r)^{1/2})N\vspace{1ex}\\
    \hspace*{18.5ex}+(1-r)^{1/2}\bar{\nabla}_\xi N.
    \end{array}
\end{equation}

To compute the second term on the right-hand side, first notice that, by \eqref{derS2},
$$
\nabla_{\xi}\xi=(tI-R)A\xi=(1-2r)\lambda\xi,
$$
where $\lambda$ is the eigenvalue of $A$ associated with $\xi$. Thus $\nabla_\xi \hat \xi=0$. Also, from \eqref{dert2} and the last equation in \eqref{eq:pi1}, we obtain
$$\xi(r)=\langle\mbox{grad}\, r, \xi\rangle=2\langle A\xi, \xi\rangle=2\lambda r(1-r),$$
hence
 \begin{equation}\label{lambda}\lambda=\frac{1}{2}r^{-1}(1-r)^{-1}\xi(r).
 \end{equation}
 Consequently,
 \begin{eqnarray*} 
 \bar \nabla_{\xi} f_*\hat \xi&=& f_*\nabla_{\xi}\hat \xi +\langle A\xi, \hat \xi\rangle N\\
 &=&\lambda r^{1/2}(1-r)^{1/2}N,
 \end{eqnarray*}
which, together with \eqref{lambda}, gives
\begin{eqnarray}\label{second}
r^{1/2} \bar{\nabla}_\xi f_*\hat \xi&=&\frac{1}{2}(1-r)^{-1/2}\xi(r)N\nonumber\\
&=& -\xi((1-r)^{1/2})N.
\end{eqnarray}
For the fourth term, we first compute, using again \eqref{lambda}: 
\begin{eqnarray*}
\bar\nabla_{\xi} N&=& -f_*A\xi\\
    &=& -\lambda f_*\xi\\
    &=& -\frac{1}{2} r^{-1}(1-r)^{-1}\xi(r)r^{1/2}(1-r)^{1/2} f_*\hat \xi.
    \end{eqnarray*}
    Hence,
    \begin{eqnarray}\label{fourth}
    (1-r)^{1/2}  \bar\nabla_{\xi} N&=&-\frac{1}{2} r^{-1/2}\xi(r) f_*\hat \xi\nonumber\\
    &=&-\xi(r^{1/2}) f_*\hat \xi.
\end{eqnarray}
Substituting \eqref{second} and \eqref{fourth} in \eqref{barnabla} shows that 
$\bar \nabla_{\xi} {\pi_2}f_*(r^{-1/2}\hat\xi)=0$, thus proving our claim.

Now fix $(x_1^0, x_2^0)\in M_1^{k-1}\times M_2^{n-k}$, and define 
$h\colon M_1^{k-1} \to \mathbb{Q}_{c_1}^{k}$ and $g \colon  M_2^{n-k} \to  \mathbb{Q}_{c_2}^{n-k+1}$ by
$$h(x_1)=(\pi_1\circ \tilde f)(0, x_1, x_2^0)\,\,\,\,\mbox{and}\,\,\,\,g(x_2)=(\pi_2\circ \tilde f)(0, x_1^0, x_2).$$
Notice that, since  $\xi$ is an eigenvector of $A$, then $\operatorname{grad} r$ is colinear with $\xi$ by \eqref{dert2}. Therefore,  $r(\gamma(s))$ depends only on $s$, that is, it does not depend on the chosen pair $(x_1^0, x_2^0)$.
Let $a,b\in C^\infty(I)$ be given by
$$a(s)=\int_0^s \sqrt{1-r(\gamma(\tau))}\,d\tau\quad \mbox{and}\quad b(s)=\int_0^s \sqrt{r(\gamma(\tau))}\,d\tau.$$
Then
\begin{eqnarray*}
    \tilde f(s,x_1, x_2)&=&((\pi_1\circ \tilde f)(s,x_1, x_2), (\pi_2\circ \tilde f)(s,x_1, x_2))\\
    &=&((\pi_1\circ \tilde f)(s,x_1, x_2^0), (\pi_2\circ \tilde f)(s,x_1^0, x_2))\\
&=&(h_{a(s)}(x_1), g_{b(s)}(x_2)),
\end{eqnarray*}
where the second equality is due to the fact that the maps $\pi_1\circ \tilde f$  and $\pi_2\circ \tilde f$ do not depend on $x_2$ and $x_1$, respectively, and the last one to the fact that the maps $s\mapsto (\pi_1\circ \tilde f)(s,x_1, x_2^0)$ and $s\mapsto (\pi_2\circ \tilde f)(s,x_1^0, x_2)$ are pregeodesics of $\mathbb{Q}_{c_1}^{k}$ and $\mathbb{Q}_{c_2}^{n-k+1}$, respectively, starting in 
$(\pi_1\circ \tilde f)(0,x_1, x_2^0)=h(x_1)$ and $(\pi_2\circ \tilde f)(0,x_1^0, x_2)=g(x_2)$, whose velocity vectors have lengths
$(1-r(\Phi(s, x_1, x_2^0))^{1/2}=(1-r(\gamma(s)))^{1/2}$ and $(r(\Phi(s, x_1^0, x_2))^{1/2}=(r(\gamma(s)))^{1/2}$, respectively. Hence $\tilde f$ is given as in \eqref{definition of f-1}, completing the proof. 
\end{proof}

\begin{example}\label{doubly} An important particular case of a hypersurface 
$$f \colon M^n=I \times M_1^{k-1} \times M_2^{n-k} \to \mathbb{Q}_{c_1}^{k} \times \mathbb{Q}_{c_2}^{n-k+1} \subset \mathbb{R}^{N_1}_{\sigma(c_1)}\times \mathbb{R}^{N_2}_{\sigma(c_2)}=\mathbb{R}^{N_1+N_2}_{\mu},$$
$2 \leq  k \leq n -1$, given by \eqref{definition of f-1}, arises when $h\colon M_1^{k-1} \to \mathbb{Q}_{c_1}^{k}$ and $g \colon  M_2^{n-k} \to  \mathbb{Q}_{c_2}^{n-k+1}$ are umbilical hypersurfaces. In what follows, we assume $c_1\neq 0\neq c_2$, the case $0=c_1\neq c_2$ being similar. Let $V_1$ and $V_2$ denote the linear subspaces of $\mathbb{R}^{N_1}_{\sigma(c_1)}$ and $\mathbb{R}^{N_2}_{\sigma(c_2)}$ that are parallel to the affine subspaces which contain $h(M_1)$ and  $g(M_2)$, respectively, and let $P^2_i$ be a two-dimensional subspace of $\mathbb{R}^{N_i}_{\sigma(c_i)}$ containing the line $V_i^\perp$, $1\leq i\leq 2$.
Let $P^4=P^2_1\oplus P_2^2\subset \mathbb{R}^{N_1+N_2}_{\mu}$, and let ${O}(P^2)$ be the subgroup of
${O}_{\sigma(c_1)}(k)\times {O}_{\sigma(c_2)}(n-k+ 1)$ that leaves $P^2=V_1^\perp\oplus V_2^{\perp}$
pointwise fixed. Then $f(M^n)$ is generated by the action of ${O}(P^2)$ on the image of the profile curve $\gamma\colon I\to \mathbb{Q}_{c_1}^{1} \times \mathbb{Q}_{c_2}^{1} \subset P^4$, given by $$\gamma(s)=(r_1C_{\epsilon_1}(a(s)/r_1), r_1S_{\epsilon_1}(a(s)/r_1), r_2C_{\epsilon_2}(b(s)/r_2), r_2S_{\epsilon_2}(b(s)/r_2)),$$
where $a,b\in C^{\infty}(I)$ satisfy \eqref{ab}. We call $f$ a  \emph{doubly rotational} hypersurface with $\gamma$ as profile. 
Doubly rotational hypersurfaces are characterized by the property that the restrictions of their shape operators to the eigenbundles of the tensor $R$ are multiples of the identity endomorphism.
\end{example}

\section{Constant mean curvature hypersurfaces 
in class  $\mathcal{A}.$}

Our next result characterizes hypersurfaces with constant mean curvature in $\mathbb{Q}_{c_1}^{k} \times \mathbb{Q}_{c_2}^{n-k+1} $, $2 \leq  k \leq n -1$, that belong to class $\mathcal{A}$ and do not split in any open subset.

\begin{theorem}\label{thm:cmc} 
Let $
h \colon M_1^{k-1} \to \mathbb{Q}_{c_1}^{k}$  and $
g \colon M_2^{n-k} \to \mathbb{Q}_{c_2}^{n-k+1}
$, $2 \leq  k \leq n -1$,
be isoparametric hypersurfaces. 
Let $h_s$ and $g_s$ denote the families of parallel hypersurfaces to $h$ and $g$, respectively, with the parameter $s$ ranging on an open interval $I=(-\delta, \delta) \subset \mathbb{R}$ where $h_s$ and $g_s$ are immersions. 
Let $H^h(s)$ and $H^g(s)$ denote the (constant) mean curvatures of $h_s$ and $g_s$, respectively.  
Given $\mathcal{H} \in \mathbb{R},$ let
$
\theta \colon I \to \mathbb{R}
$  
be a smooth function such that $\theta^{-1}(\{k\pi/2\,:\, k\in \mathbb{Z}\})$ has empty interior and
\begin{equation}\label{eq: differential equation a}
\theta'(s) - (k-1)\sin\theta(s) \, 
H^h(a(s)) + (n-k)\cos \theta(s)  
H^g(b(s)) = n\mathcal{H}  
\end{equation}
for all  $s \in I$, where
$$  a(s)= \int_0^s\cos \theta(u)du\quad \text{and} \quad b(s) = \int_{0}^s \sin\theta(u) \, du.$$
Then the restriction of the map  
$$
f \colon M^n = I \times M_1^{k-1} \times M_2^{n-k} \to \mathbb{Q}_{c_1}^{k} \times \mathbb{Q}_{c_2}^{n-k+1},  
$$  
defined by 
\begin{equation*}
f(s,x_1,x_2) = h_{a(s)}(x_1) +  g_{b (s)}(x_2),
\end{equation*}
to the open subset of its regular points, is a hypersurface with constant mean curvature $\mathcal{H}$ in class $\mathcal{A}$ that does not split in any open subset.
    
Conversely, any  hypersurface 
$
f \colon M^{n} \to \mathbb{Q}_{c_1}^{k} \times \mathbb{Q}_{c_2}^{n-k+1}
$, $2 \leq k \leq n-1$,  with constant mean curvature $\mathcal{H}$ that is in class $\mathcal{A}$ and does not split in any open subset is given locally 
 in this way in an open and dense subset of $M^n$.
\end{theorem}

\begin{proof}
 It follows from \eqref{eq1 - derivada da f} and \eqref{A-eigenspaces} that 
\begin{eqnarray*}n\mathcal{ H}^f(s, x_1, x_2)&=&
\theta'(s) - (k-1)\sin \theta (s)\, \mathcal{ H}^{h_a}(x_1)  
  + (n-k)\cos \theta(s)\, \mathcal{ H}^{g_b}(x_2)\\
  &=&
\theta'(s) - (k-1)\sin \theta(s) H^h(a(s))+(n-k)\cos \theta(s) H^g(b(s))
\end{eqnarray*}
for all $(s, x_1, x_2)\in M^n$.
Thus $f$ has constant mean curvature  $\mathcal{H}$ by  \eqref{eq: differential equation a}. The other assumption on $\theta$ and \eqref{R at dt} imply that $f$ does not split in any open subset.

Conversely, let $f \colon M^{n} \to \mathbb{Q}_{c_1}^{k} \times \mathbb{Q}_{c_2}^{n-k+1}$, $2 \leq  k \leq n -1$, be a  hypersurface with constant mean curvature $\mathcal{H}$ that is in class $\mathcal{A}$ and does not split in any open subset. By Theorem \ref{classAcylinders}, $f$ is locally given by \eqref{definition of f-1}, and by  \eqref{eq1 - derivada da f} and \eqref{A-eigenspaces} we have 
\begin{equation}\label{eq: mean curvature of h and g}
\begin{split}
  n\mathcal{ H}= \theta'(s) -(k-1) \sin \theta (s)\, \mathcal{ H}^{h_a}(x_1)  
  + (n-k) \cos \theta(s)\, \mathcal{ H}^{g_b}(x_2)
\end{split}
\end{equation}
for any $x=(s,x_1,x_2) \in M^n = I \times M_1^k \times M_2^{n-k}$. 
Thus, for each $s\in I$,  the hypersurfaces
$$
h_{a(s)} \colon M_1^{k-1} \to \mathbb{Q}_{c_1}^{k} \quad \text{and} \quad
g_{b(s)} \colon M_2^{n-k} \to \mathbb{Q}_{c_2}^{n-k+1}
$$
have constant mean curvature. 
Since $a'(s)$ and $b'(s)$ are positive functions and  $a(0)=b(0)=0$, there exists  
$
\delta >0$ such that $(-\delta, \delta )\subset a(I)\cap b(I)
$, hence for any 
$
s\in (-\delta , \delta )
$ 
one can choose $s_1, s_2 \in I$ with  $a(s_1)=s=b(s_2)$. It follows that $h_s$ and $g_s$ have constant mean curvature for all $s\in (-\delta, \delta)$, hence both $h$ and $g$ are isoparametric by a well-known theorem by Cartan.  For any $s \in (-\delta, \delta ),$  denote by $H^h(s)$ and $H^g(s)$ the mean curvature of $h_s$ and $g_s,$ respectively. 
Then  \eqref{eq: mean curvature of h and g} becomes \eqref{eq: differential equation a}, and the proof is completed.
\end{proof}

\section{Hypersurfaces of  $\mathbb{Q}_{c_1}^{k} \times \mathbb{Q}_{c_2}^{n-k+1} $ in class  $\mathcal{A}$ with constant product angle function}

Given a hypersurface $f\colon M^n\to \mathbb{Q}_{c_1}^{k} \times \mathbb{Q}_{c_2}^{n-k+1} $, its \emph{product angle function} is defined by $\Theta=\langle PN, N\rangle$, where $N$ is a unit normal vector field along $f$ and $P$ is the product structure of $\mathbb{Q}_{c_1}^{k} \times \mathbb{Q}_{c_2}^{n-k+1}$, given by $P(X_1, X_2)=(X_1, -X_2)$
for tangent vector fields $X_1$ and $X_2$. Recall from the proof of Proposition~\ref{product locally splits} that $\Theta=2r-1$, hence
 $\Theta$ is constant if and only if the same holds for $r$, and that $f$ splits locally if and only if $\Theta$ has a constant value $\Theta_0=\pm 1$.
Moreover, it follows from \eqref{dert2} that, under the assumption of constancy of~$\Theta$ (or, equivalently, of $r$ or $t$), the vector field $\xi$ is a principal direction of $f$ with $0$ as the associated principal curvature. Observe that, unlike the case where the ambient space is $\mathbb{S}^n\times \mathbb{R}$ or $\mathbb{H}^n\times \mathbb{R}$, in which case hypersurfaces with constant angle are automatically in class $\mathcal{A}$, this may not be the case if the ambient space is a product $\mathbb{Q}_{c_1}^{k} \times \mathbb{Q}_{c_2}^{n-k+1} $ with $2\leq k\leq n-1$, as shown, for instance, by some minimal or constant curvature $3$-dimensional hypersurfaces of $\mathbb{S}^2\times \mathbb{S}^2$ (see Theorem $1.1$ of \cite{LVWY}). 

For a hypersurface given by \eqref{definition of f-1}, since $r=\sin^2\theta$ by \eqref{RR}, the product angle function $\Theta$ is constant if and only if the function $\theta(s)$ is constant. Thus, Theorem \ref{classAcylinders} has the following immediate consequence. 

  \begin{corollary}\label{cor: PAC} Let $h_s\colon M_1^{k-1}\to \mathbb{Q}_{c_1}^{k}$ and  $g_s \colon  M_2^{n-k} \to \mathbb{Q}_{c_2}^{n-k+1}$, $2 \leq  k \leq n -1$, be parallel families of hypersurfaces of $\mathbb{Q}_{c_1}^{k}$ and  $\mathbb{Q}_{c_2}^{n-k+1}$, respectively.  
Then, for any $\theta\in (0, 2\pi)\setminus \{\pi/2, \pi, 3\pi/2\}$, the restriction of the map  
$$
 f \colon M^n=\mathbb{R} \times M_1^{k-1} \times M_2^{n-k} \to \mathbb{Q}_{c_1}^{k} \times \mathbb{Q}_{c_2}^{n-k+1} \subset \mathbb{R}^{N_1}_{\sigma(c_1)}\times \mathbb{R}^{N_2}_{\sigma(c_2)}=\mathbb{R}^{N}_{\mu},$$
 defined by  
\begin{equation}\label{definition of f}
f(s,x_1,x_2) = h_{s \cos \theta}(x_1) +  g_{s\sin \theta}(x_2)
\end{equation}
for all $(s,x_1,x_2) \in M^n$, to the open subset of its regular points, defines a hypersurface in class $\mathcal{A}$ that has constant product angle function $\Theta=-\cos (2\theta)$ (which therefore does not split in any open subset).

Conversely, any hypersurface $f\colon M^n\to \mathbb{Q}_{c_1}^{k} \times \mathbb{Q}_{c_2}^{n-k+1} $, $2 \leq  k \leq n -1$, in class $\mathcal{A}$ that has constant product angle function
and does not split in any open subset is given locally in this way in an open and dense subset of $M^n$. 
\end{corollary}

\subsection{CMC hypersurfaces 
in class  $\mathcal{A}$ with constant product angle function}

In this last subsection, we classify the hypersurfaces of  $\mathbb{Q}_{c_1}^{k} \times \mathbb{Q}_{c_2}^{n-k+1} $, $2 \leq  k \leq n -1$, that are in class  $\mathcal{A}$ and have constant mean curvature and a constant product angle function.

First notice that, if $f\colon M^n\to \mathbb{Q}_{c_1}^{k} \times \mathbb{Q}_{c_2}^{n-k+1}$  splits, for instance, if  $M^n$ is isometric to the product $\mathbb{Q}^k_{c_1} \times M_2^{\,n-k}$ and  
$f = i_1 \times f_2$, where  
$i_1 \colon \mathbb{Q}^k_{c_1} \to \mathbb{Q}^k_{c_1}$ is the identity and  
$f_2 \colon M_2^{\,n-k} \to \mathbb{Q}_{c_2}^{\,n-k+1}$ is a hypersurface,  then $f$ has a constant product angle function $\Theta=-1$, and $f$ has constant mean curvature if and only if the same holds for $f_2$.  Similarly, if $M^n$ is isometric to  
$M_1^{\,k-1} \times \mathbb{Q}_{c_2}^{\,n-k+1}$ and $f = f_1 \times i_2$, where $i_2 \colon \mathbb{Q}^{n-k+1}_{c_2} \to \mathbb{Q}^{n-k+1}_{c_2}$ is the identity and  
$f_1 \colon M_1^{\,k-1} \to \mathbb{Q}_{c_1}^{\,k}$ is a hypersurface, then $f$ 
has a constant product angle function $\Theta=1$ and $f$ has constant mean curvature if and only if the same holds for $f_1$.

From now on, we consider hypersurfaces $f\colon M^n\to \mathbb{Q}_{c_1}^{k} \times \mathbb{Q}_{c_2}^{n-k+1}$, $2 \leq  k \leq n -1$, with constant mean curvature   $\mathcal{H}$ whose product angle function is a constant $\Theta\in (-1, 1)$, and therefore do not split in any open subset by Proposition \ref{product locally splits}.  
By Theorem \ref{thm:cmc} and Corollary~\ref{cor: PAC}, there exist isoparametric hypersurfaces 
$h \colon M_1^{k-1} \to \mathbb{Q}_{c_1}^{k}$ and $g \colon M_2^{n-k} \to \mathbb{Q}_{c_2}^{n-k+1}$,
and $\theta \in (0, 2\pi)\setminus \{\pi/2, \pi, 3\pi/2\}$,
such that $f$ is locally given by \eqref{definition of f}. Eq.~\eqref{eq: differential equation a} reduces to
\begin{equation}\label{eq: constant theta}
- (k-1)\sin\theta\, \mathcal{H}^h\!\left( s\cos \theta  \right) 
+ (n-k)\cos \theta\, \mathcal{H}^g\!\left( s\sin \theta  \right) 
= n\mathcal{H}, 
\quad  s \in (-\delta, \delta),
\end{equation}
where $\mathcal{H}^h(s)$ and $\mathcal{H}^g(s)$ denote the mean curvatures of $h_s$ and $g_s$, respectively.

The following example shows that each isoparametric hypersurface of $\mathbb{Q}_{c}^{k}$, $ k \ge 2$,   gives rise to a minimal hypersurface in class $\mathcal{A}$ of $\mathbb{Q}_{c}^{k} \times \mathbb{Q}_{c}^{k}$ with a vanishing product angle function.

\begin{example} \label{example1}
Let $
h,g \colon M^{k-1} \to \mathbb{Q}_{c}^{k}$, $ k \ge 2$,  
be congruent isoparametric hypersurfaces, that is, $\Phi \circ h = g$ for some isometry $\Phi \colon \mathbb{Q}_{c}^{k} \to \mathbb{Q}_{c}^{k}$. Let \(N^h\) and \(N^g\) be unit normal vector fields along $h$ and $g$, respectively, and let
 $\theta\in \{\tfrac{\pi}{4},\tfrac{5\pi}{4}\}$ if $\Phi_* N^h = N^g$,  and $\theta\in\{\tfrac{3\pi}{4},\tfrac{7\pi}{4}\}$ if $\Phi_* N^h = -N^g$.
Let $\delta > 0$ be such that $h_{s \cos \theta}$  and $g_{s\sin\theta} $
are immersions for all $s \in (-\delta, \delta)$.
Then the map
$
f \colon (-\delta, \delta) \times M^{k-1} \times M^{k-1} \to \mathbb{Q}_c^{k} \times \mathbb{Q}_c^{k},
$
defined by \eqref{definition of f}, is a minimal hypersurface in class~$\mathcal{A}$ with a vanishing constant product angle function.
\end{example}

It is well known that an isoparametric hypersurface of either Euclidean or hyperbolic space can have at most two distinct principal curvatures. For an  isoparametric hypersurface $h\colon M^n\to \mathbb{H}_c^{n+1}$, $c=-1/r^2$, with a single principal curvature $\lambda \geq 0$,
\begin{equation}\label{hip1}
\mathcal{H}^h(s)=\begin{cases}
1/r, & \text{if } \lambda =1/r,\\
\frac{1}{r}\tanh \left(\frac{\varphi -s}{r}\right), & \text{if } \lambda=\frac{1}{r}\tanh \left(\frac{\varphi}{r}\right),\\
\frac{1}{r}\coth \left(\frac{\varphi -s}{r}\right), & \text{if } \lambda = \frac{1}{r}\coth \left(\frac{\varphi}{r}\right),
\end{cases}
\end{equation}
whereas
\begin{equation}\label{hip2}n\mathcal{H}^h(s)= \frac{m}{r}  \tanh  \left(\frac{\varphi -s}{r}\right) + \frac{n-m}{r} \coth  \left(\frac{\varphi -s}{r}\right)\end{equation}
if $h$ has two distinct principal curvatures $\lambda_1=\frac{1}{r}\tanh \left(\frac{\varphi}{r}\right)$ and $\lambda_2=\frac{1}{r}\coth \left(\frac{\varphi}{r}\right)$, with multiplicities $m$ and $n-m$, respectively. 

If $h\colon M^n\to \mathbb{R}^{n+1}$ is an  isoparametric hypersurface, then either $h$ is totally geodesic or 
\begin{equation}\label{euc}
n\mathcal{H}^h(s)=\frac{m\lambda}{1-s\lambda},
\end{equation}
where $\lambda$ is the unique nonzero principal curvature of $h$, with multiplicity $m$.

For an isoparametric hypersurface $h\colon M^n \to \mathbb{S}_c^{n+1}$, $c=1/r^2$, the following facts are known:
\begin{itemize}
    \item The number $\ell$ of distinct principal curvatures satisfies $\ell\in \{1,2,3,4, 6\}$;
    \item The distinct principal curvatures are $\frac{1}{r}\cot (\varphi_1/r), \ldots, \frac{1}{r}\cot (\varphi_\ell/r)$, with $\varphi_i=\varphi_1+(i-1)\frac{r\pi}{\ell}$, $1\leq i\leq \ell$
    (called the \emph{principal curvature angles}), and their multiplicities satisfy $m_{k+2}=m_k \,(\mbox{mod} \,\ell)$ (see \cite{M}). In particular, the principal curvatures can have at most two distinct multiplicities.
    \item If there are $\ell$ distinct principal curvatures, all of which of the same multiplicity $m_\ell$, then $m_\ell\in \{1,2,4,8\}$ if $\ell=3$, and $m_\ell\in \{1,2\}$ if $\ell=4$ or $\ell=6$.
    \item The principal curvatures can have two distinct multiplicities only if $\ell =2$ or $\ell=4$ (see, e.g., \cite{CR})).  
\end{itemize}
    
If $h\colon M^n \to \mathbb{S}_c^{n+1}$, $c=1/r^2$, is an isoparametric hypersurface with principal curvatures $\frac{1}{r}\cot (\varphi_1/r), \ldots, \frac{1}{r}\cot (\varphi_\ell/r)$, of multiplicities $m_1, \ldots, m_\ell$, respectively, then
\begin{equation}\label{sph}\
n\mathcal{H}^h(s)
=\frac{1}{r}\sum_{i=1}^{\ell} m_i\,\cot\left(\frac{\varphi_i - s}{r}\right).
\end{equation}

To determine all possible cases in which \eqref{eq: constant theta} is satisfied, we make use of the fact that the functions $z\in \mathbb{C}\mapsto \mathcal{H}^h(z\cos\theta)$ and $z\in \mathbb{C}\mapsto \mathcal{H}^g(z\sin\theta)$ are analytic in the complement $\mathbb{C}\setminus U_{\mathbb{C}}$ of the subset $U_{\mathbb{C}}$ of their singularities; therefore, \eqref{eq: constant theta} implies that
\[
- (k-1)\sin\theta\, \mathcal{H}^h\!\left( z\cos \theta  \right) 
+ (n-k)\cos \theta\, \mathcal{H}^g\!\left( z\sin \theta  \right) 
= n\mathcal{H}, 
\quad z\in \mathbb{C}\setminus U_{\mathbb{C}},
\]
or equivalently,
\begin{equation}\label{eq: constant theta 2}
- (k-1)\tan\theta\, \mathcal{H}^h\!\left( z  \right) 
+ (n-k) \mathcal{H}^g\!\left( z\tan \theta  \right) 
= n\mathcal{H}\sec \theta, 
\quad  z\in \mathbb{C}\setminus U_{\mathbb{C}}.
\end{equation}

In the sequel, we study separately, for each pair $(c_1, c_2)$, the possible cases in which the above equation is satisfied, starting with the most interesting case where $c_1, c_2>0$.

\subsubsection{The case $c_1, c_2>0$.} Set $c_i=1/r_i^2$, $r_i>0$, $i=1,2$. 
Let  $\lambda_i = \frac{1}{r_1} \cot( \frac{\varphi_i}{r_1})$, $1\leq i\leq \ell_1$, and 
$\mu_j = \frac{1}{r_2} \cot( \frac{\psi_j}{r_2})$, $1\leq j\leq \ell_2$, 
be the distinct principal curvatures of the isoparametric hypersurfaces $h \colon M_1^{k-1} \to \mathbb{S}_{c_1}^k$ and $g \colon M_2^{n-k} \to  \mathbb{S}_{c_2}^{n-k+1}$, respectively. Then the mean curvatures of the parallel hypersurfaces 
$h_s$ and $g_s$ are, respectively,
\[
\mathcal{ H}^{h_s}(x) = \frac{1}{r_1(k-1)} \sum_{i=1}^{\ell_1} m_i\cot\left( \frac{\varphi_i - s}{r_1}\right) 
\]
and
\[
\mathcal{ H}^{g_b}(x) = \frac{1}{r_2(n-k)} \sum_{j=1}^{\ell_2} p_j\cot \left( \frac{\psi_j - s}{r_2}\right).
\]
Eq. \eqref{eq: constant theta} now takes the form
\begin{equation} \label{real}
-\frac{\sin \theta}{r_1} \sum_{i=1}^{\ell_1} m_i\cot\left( \frac{\varphi_i - s \cos \theta}{r_1}\right) +  \frac{\cos \theta}{r_2} \sum_{j=1}^{\ell_2}p_j \cot \left( \frac{\psi_j - s \sin \theta}{r_2}\right) = n\mathcal{H}
\end{equation}
for all $s \in (-\delta, \delta)$, and its complexified version \eqref{eq: constant theta 2} becomes
\begin{equation}\label{eq:r1r2}
\begin{aligned}
&-\frac{\tan\theta}{r_1}
\sum_{i=1}^{\ell_1}m_i
\cot\left(\frac{\varphi_i-z}{r_1}\right)+\frac{1}{r_2}
\sum_{j=1}^{\ell_2}p_j
\cot\left(\frac{\psi_j-z\tan\theta}{r_2}\right)
=n\mathcal{H}\sec \theta
\end{aligned}
\end{equation}
for any $z \in \mathbb{C} \setminus U_{\mathbb{C}}$, where  
\[
U_{\mathbb{C}}
=
\left\{
\varphi_i+r_1d\pi,\,
(\psi_j+r_2d\pi)\cot\theta
:
1\leq i\leq \ell_1,\,
1\leq j\leq \ell_2,\,
d\in\mathbb{Z}
\right\}.
\]


\begin{lemma} \label{gen1}
If \eqref{eq:r1r2} holds, then $\mathcal{H}=0$,
\begin{equation} \label{equalities}
|\tan\theta|=\frac{\ell_1r_2}{\ell_2r_1} =\frac{(n-k) r_1}{(k-1) r_2}\quad \mbox{and}\quad \frac{\ell_1}{r_1}\varphi_1=\left\{ \begin{array}{l}\frac{\ell_2}{r_2} \psi_1,\quad\,\,\,\,\,\,\,\mbox{if}\,\,\tan\theta>0\\
\pi-\frac{\ell_2}{r_2} \psi_1,\,\,\mbox{if}\,\,\tan\theta<0 \end{array}\right.
\end{equation}
\end{lemma}
\begin{proof}
It follows from
\begin{equation}
    \cot(B+is)=
\frac{\sin 2B-i\sinh 2s}{\cosh 2s-\cos 2B}, \quad \text{ for all } B\in \mathbb{R}  \text{ and } s\neq 0,
\end{equation}
that \(\lim_{s\to +\infty}   \cot(B+ias) = -\operatorname{sgn}(a) i\)  for every $a\neq 0$.
Thus, setting $z=is$ in \eqref{eq:r1r2} and taking $s \to \infty$  give
\begin{equation}
\left( -\frac{\tan\theta}{r_1} \sum_{i=1}^{\ell_1} m_i \,  
+ \frac{\operatorname{sgn}(\tan \theta)}{r_2}\sum_{j=1}^{\ell_2} p_j  \right)i
= n\mathcal{H} \sec \theta.\end{equation}
Consequently, $ \mathcal{H}=0$ and $| \tan \theta |=\frac{(n-k) r_1}{(k-1) r_2}.$ Writing $a=\tan \theta$,  then \eqref{eq:r1r2} reduces to 
\begin{equation}\label{eq:ct}
\frac{a}{r_1}
\sum_{i=1}^{\ell_1}m_i
\cot\left(\frac{\varphi_i-z}{r_1}\right)
=\frac{1}{r_2}
\sum_{j=1}^{\ell_2}p_j
\cot\left(\frac{\psi_j-az}{r_2}\right),
\end{equation}
for any $z \in \mathbb{C} \setminus U_{\mathbb{C}}$. 
Using the fact that the principal curvature angles $0<\varphi_1<\cdots <\varphi_{\ell_1}$ and 
$0<\psi_1<\cdots<\psi_{\ell_2}$ satisfy $\varphi_i=\varphi_1+(i-1)\frac{\pi r_1}{\ell_1}$, $1\leq i\leq \ell_1$, and 
 $\psi_j=\psi_1+(j-1)\frac{\pi r_2}{\ell_2}$, $1\leq j\leq \ell_2$, then \eqref{eq:ct} becomes 
\begin{equation}\label{eq:r1r2b}
\frac{a}{r_1}
\sum_{i=1}^{\ell_1}m_i
\cot\left(\frac{\varphi_1-z}{r_1} + \frac{(i-1)\pi}{\ell_1}\right)
=\frac{1}{r_2}
\sum_{j=1}^{\ell_2}p_j
\cot\left(\frac{\psi_1-az}{r_2} + \frac{(j-1)\pi}{\ell_2}\right)
\end{equation}
for any $z \in \mathbb{C} \setminus U_{\mathbb{C}}$. We consider both
sides of \eqref{eq:r1r2b} as meromorphic functions on $\mathbb C$. Using the fact that the poles of $\cot w$ are 
$w=q\pi$, $q\in\mathbb Z$, the sets of poles of the left and right sides of \eqref{eq:r1r2b} are, respectively,
\[
P_L=
\left\{
\varphi_1+\frac{d\pi r_1}{\ell_1}: d\in\mathbb Z
\right\}
\quad \mbox{and}\quad P_R=\left\{
\frac{\psi_1}{a}
+\frac{d\pi r_2}{\ell_2a}:d\in\mathbb Z
\right\}.
\]
  Since the two meromorphic functions in \eqref{eq:r1r2b} are identical, then $P_L=P_R$, and hence
\[
\frac{\pi r_1}{\ell_1}\mathbb Z=P_L-P_L=P_R-P_R=\frac{\pi r_2}{\ell_2a}\mathbb Z.
\]
Thus,
\[\frac{\pi r_1}{\ell_1}=\frac{\pi r_2}{\ell_2|a|},\]
and the first equality in \eqref{equalities} follows. Writing $c={\pi r_1}/{\ell_1}={\pi r_2}/{\ell_2|a|}$, from 
\[
\varphi_1+c\mathbb Z=P_L=P_R=\frac{\psi_1}{a}+c\mathbb Z
\]
it follows that $\varphi_1-\frac{\psi_1}{a}\in c\mathbb{Z}$.
Since $0<\varphi_1<\frac{\pi r_1}{\ell_1}$ and  $0<\psi_1<\frac{\pi r_2}{\ell_2}$, if $a>0$ then  $0<\varphi_1-\frac{\psi_1}{a}<\frac{\pi r_1}{\ell_1}$, hence 
$\varphi_1-\frac{\psi_1}{a}=0$. If $a<0$, then $-\frac{\pi r_1}{\ell_1}<\frac{\psi_1}{a}<0$,
thus $0<\varphi_1-\frac{\psi_1}{a}<\frac{2\pi r_1}{\ell_1}$, and hence $\varphi_1-\frac{\psi_1}{a}=\frac{\pi r_1}{\ell_1}$. Therefore, $\ell_1 r_2\varphi_1+\ell_2 r_1\psi_1=\pi r_1r_2 $.
\end{proof}

\begin{lemma} \label{single} If the multiplicities satisfy 
\[m_1=\cdots=m_{\ell_1}:=m\quad \mbox{and}\quad p_1=\cdots=p_{\ell_2}:=p,
\] 
then \eqref{eq:r1r2b} holds if and only 
\begin{equation}\label{rel}
|\tan\theta|=\frac{\ell_1r_2}{\ell_2r_1}, \quad 
\frac{\ell_1}{r_1}\varphi_1=\left\{ \begin{array}{l}\frac{\ell_2}{r_2} \psi_1,\quad\,\,\,\,\,\,\,\mbox{if}\,\,\tan\theta>0\\
\pi-\frac{\ell_2}{r_2} \psi_1,\,\,\mbox{if}\,\,\tan\theta<0 \end{array}\right.\quad \mbox{and}\quad\frac{m\ell_1^2}{p\ell_2^2}=\frac{c_2}{c_1}.
\end{equation}
\end{lemma}
\begin{proof} 
Under the assumptions on the multiplicities, we have $k-1=m\ell_1$ and $n-k=p\ell_2$. Hence, the third equality in \eqref{rel} follows from the equality $\frac{\ell_1r_2}{\ell_2r_1}=\frac{(n-k) r_1}{(k-1) r_2}$ in \eqref{equalities}. 

Conversely, using the identity $\sum_{k=1}^n \cot\left(x+\frac{(k-1)\pi}{n}\right)=n\cot (nx)$, then  \eqref{eq:r1r2b} reduces to
\begin{equation}\label{reduced}
\frac{ma\ell_1}{r_1}\cot\left(\frac{\ell_1 (\varphi_1 - z)}{r_1}\right)=\frac{p\ell_2}{r_2}\cot\left(\frac{\ell_2 (\psi_1 -a z)}{r_2}\right),
\end{equation}
 which is an identity under the equalities in \eqref{rel}.
\end{proof}

 \begin{lemma}\label{l1=2}
If \eqref{eq:r1r2b} holds with $\ell_2\in \{1, 3, 6\}$ and either $\ell_1=2$ or $\ell_1=4$, 
then $m_1= m_2$ if  $\ell_1=2$, and  $m_1= \cdots= m_4$ if  $\ell_1=4$.
\end{lemma}
\begin{proof} 
Since
$p_1=\cdots=p_{\ell_2}:=p$ when $\ell_2\in \{1, 3, 6\}$, the right-hand side of \eqref{eq:r1r2b}  is
\begin{equation}\label{right}
    \frac{1}{r_2}
\sum_{j=1}^{\ell_2}p_j
\cot\left(\frac{\psi_1-az}{r_2} + \frac{(j-1)\pi}{\ell_2}\right)= \frac{p\ell_2}{r_2}\cot \left(\frac{\ell_2(\psi_1-az)}{r_2}\right).\\
\end{equation}
 On the other hand, the left-hand side is equal to
\begin{equation}\label{eq:l1=2}
   \frac{a}{r_1}\left(m_1\cot\left(\frac{\varphi_1-z}{r_1} \right)-m_2\tan \left(\frac{\varphi_1-z}{r_1} \right)\right)
\end{equation}
if  $\ell_1=2$, and to 
\begin{equation}\label{eq:l1=4}
 \frac{2a}{r_1} \left(m_1\cot\left(\frac{2(\varphi_1-z)}{r_1} \right)-m_2\tan \left(\frac{2(\varphi_1-z)}{r_1} \right)\right) 
\end{equation}
if  $\ell_1=4$, since $m_3=m_1$ and $m_4=m_2$. 

The residues at the poles of the meromorphic function in \eqref{eq:l1=2} or \eqref{eq:l1=4} alternate between $-am_1$ and $-am_2$, while those of the meromorphic function in \eqref{right} are all equal to $-p/a$. Therefore $m_1=m_2$.
\end{proof}

 \begin{lemma}\label{distinct} Eq. \eqref{eq:r1r2b} holds with  $\ell_1, \ell_2\in \{2,4\}$ and respective pairs of 
 distinct multiplicities $(m_1, m_2)$ and $(p_1, p_2)$ if and only if one of the following holds:
\begin{enumerate}
\item[(i)]
$
\tan \theta=\frac{\ell_1r_2}{\ell_2r_1}$, $
\frac{\ell_1\varphi_1}{r_1}
=
\frac{\ell_2\psi_1}{r_2}$ and $
\tan^2 \theta(m_1,m_2)=(p_1,p_2)
$;\vspace{1ex}
\item[(ii)]
$
\tan \theta=-\frac{\ell_1r_2}{\ell_2r_1}$, $
\frac{\ell_1\varphi_1}{r_1}
+
\frac{\ell_2\psi_1}{r_2}
=
\pi$, and 
$
\tan^2 \theta (m_1,m_2)=(p_2,p_1).
$
\end{enumerate}
\end{lemma}
\begin{proof} 
By \eqref{eq:l1=2} and \eqref{eq:l1=4}, the left-hand side of  \eqref{eq:r1r2b} is equal to
\begin{equation}\label{leftb} \frac{a\ell_1}{2r_1} \left(m_1\cot\left(\frac{\ell_1(\varphi_1-z)}{2r_1} \right)-m_2\tan \left(\frac{\ell_1(\varphi_1-z)}{2r_1} \right)\right) 
\end{equation}
Accordingly, the right-hand side is equal to
\begin{equation}\label{rightb}
\frac{\ell_2}{2r_2}\left(p_1\cot\left(\frac{\ell_2(\psi_1-az)}{2r_2} \right)-p_2\tan \left(\frac{\ell_2(\psi_1-az)}{2r_2} \right)\right).
\end{equation}

It was shown in Lemma \ref{gen1} that the common set of poles of the meromorphic functions in \eqref{leftb} and \eqref{rightb} is $\{\varphi_1 + \frac{q \pi r_1}{\ell_1}\,:\,q \in \mathbb{Z}\}$. 
The residue of the meromorphic function in \eqref{leftb} at a pole $\varphi_1 + \frac{q \pi r_1}{\ell_1}$, $q\in \mathbb{Z}$,  is equal to 
 $-am_1$ if $q$ is even, and to $-am_2$ if $q$ is odd. Since $\frac{\ell_1}{r_1}\varphi_1=\frac{\ell_2}{r_2} \psi_1 $ if $a>0$ and 
 $\frac{\ell_1}{r_1}\varphi_1=\pi-\frac{\ell_2}{r_2} \psi_1$ if $a<0$ by Lemma \ref{gen1}, the residue of the meromorphic function in \eqref{rightb} at the same pole
 is equal to 
 $-p_1/a$ if $q$ is odd, and to $-p_2/a$ if $q$ is even when $a>0$, and is equal to 
 $-p_1/a$ if $q$ is even, and to $-p_2/a$ if $q$ is odd when $a<0$. 
Hence,   $a^2(m_1, m_2)=(p_1, p_2)$ if $a>0$, and $a^2(m_1, m_2)=(p_2, p_1)$ if $a<0$. 
 Conversely, if either of the conditions $(i)$ or $(ii)$ holds, then \eqref{leftb} and \eqref{rightb} coincide, hence  \eqref{eq:r1r2b} is satisfied.
\end{proof}

We are now in a position to state and prove the following theorem.

\begin{theorem}\label{prodspheres} Let $h \colon M_1^{k-1} \to \mathbb{S}_{c_1}^{k}$ and $g \colon M_2^{n-k} \to \mathbb{S}_{c_2}^{n-k+1}$, $c_i=1/r_i^2$, $1\leq i\leq 2$,
be isoparametric hypersurfaces with  $\ell_1$ and $\ell_2$ distinct principal curvature angles $0<\varphi_1<\cdots <\varphi_{\ell_1}$ and $0<\psi_1<\cdots <\psi_{\ell_2}$, respectively, with 
\begin{equation}\label{smallest}\frac{\ell_1}{r_1}\varphi_1=\frac{\ell_2}{r_2}\psi_1\quad \mbox{or} \quad  \frac{\ell_1}{r_1}\varphi_1=\pi-\frac{\ell_2}{r_2}\psi_1.
\end{equation}
Suppose, in addition, that one of the following conditions holds:
\begin{itemize}
    \item[(i)]  the principal curvatures of $h$ and $g$ have common multiplicities $m$ and~$p$, respectively, satisfying
\begin{equation}\label{relb}
\frac{m\ell_1^2}{p\ell_2^2}=\frac{c_2}{c_1}.
\end{equation}
\item[(ii)]  $\ell_1, \ell_2\in \{2,4\}$ and the principal curvatures of $h$ and $g$  have respective pairs of 
 distinct multiplicities $(m_1, m_2)$ and $(p_1, p_2)$ such that \vspace{1ex}
 
 \begin{itemize}
      \item[(a)] $\ell^2_1c_1(m_1, m_2)=\ell^2_2c_2(p_1, p_2)$ if $\frac{\ell_1}{r_1}\varphi_1=\frac{\ell_2}{r_2}\psi_1$, or \vspace{1ex}
     \item[(b)]   $\ell^2_1c_1(m_1, m_2)=\ell^2_2c_2(p_2, p_1)$ if $\frac{\ell_1}{r_1}\varphi_1=\pi-\frac{\ell_2}{r_2}\psi_1$.
\end{itemize}
\end{itemize}
Then
 \[f\colon M^n= (-\delta, \delta) \times M_1^{k-1} \times M_2^{n-k}\to \mathbb{S}_{c_1}^{k}\times \mathbb{S}_{c_2}^{n-k+1},\]  
 given by \eqref{definition of f} with 
 \begin{equation}\label{tan}
 \tan \theta=\left\{\begin{array}{l}\frac{\ell_1r_2}{\ell_2r_1}\quad \mbox{if} \quad  \frac{\ell_1}{r_1}\varphi_1=\frac{\ell_2}{r_2}\psi_1,\vspace{1ex}\\
 -\frac{\ell_1r_2}{\ell_2r_1}\quad \mbox{if} \quad   \frac{\ell_1}{r_1}\varphi_1=\pi-\frac{\ell_2}{r_2}\psi_1,\end{array} \right.
  \end{equation}
  is a minimal hypersurface in class $\mathcal{A}$ with a constant product angle function
  \begin{equation}\label{paf}
  \Theta=\frac{c_1\ell_1^2-c_2\ell_2^2}{c_1\ell_1^2+c_2\ell_2^2}.
  \end{equation}

 Conversely, any hypersurface $f\colon M^n\to \mathbb{S}_{c_1}^{k}\times \mathbb{S}_{c_2}^{n-k+1}$ 
 in class $\mathcal{A}$ with constant mean curvature and a constant product angle function $\Theta\in (-1, 1)$ is minimal and given locally as above.
 \end{theorem} 
 \begin{proof} 
 By Corollary \ref{cor: PAC}, $f$ defines a hypersurface in class $\mathcal{A}$ with a constant product angle function $\Theta=-\cos(2\theta)=({c_1\ell_1^2-c_2\ell_2^2})/({c_1\ell_1^2+c_2\ell_2^2})$. If condition $(i)$ (respectively, $(ii)$) holds, it follows from  Lemma \ref{single} (respectively, Lemma \ref{distinct}) that \eqref{eq:r1r2b} is satisfied. Hence \eqref{eq:r1r2}, and therefore \eqref{real}, are valid with $\mathcal{H}=0$. Thus, also \eqref {eq: constant theta} is satisfied with $\mathcal{H}=0$, hence  $f$ is minimal.

   Conversely, let $f\colon M^n\to \mathbb{S}_{c_1}^{k}\times \mathbb{S}_{c_2}^{n-k+1}$ be a hypersurface
 in class $\mathcal{A}$ with constant mean curvature and  constant product angle function $\Theta\in (-1, 1)$. By Corollary \ref{cor: PAC}, 
 there exist isoparametric hypersurfaces
$h \colon M_1^{k-1} \to \mathbb{S}_{c_1}^{k}$ and
$g \colon M_2^{n-k} \to \mathbb{S}_{c_2}^{n-k+1}$, and a constant
$\theta \in (0,2\pi) \setminus
\left\{ \tfrac{\pi}{2}, \pi, \tfrac{3\pi}{2} \right\}$,
such that, locally,
$M^n = (-\delta, \delta) \times M_1^{k-1} \times M_2^{n-k}$ and
$f$ is given by \eqref{definition of f}. 
Then,  the mean curvatures $\mathcal{H}^h(s)$ and $\mathcal{H}^g(s)$ of $h_s$ and $g_s$, respectively, 
satisfy \eqref{eq: constant theta}. Equivalently, \eqref{real}, hence its complexified version \eqref{eq:r1r2},
is satisfied by the principal curvature angles $\varphi_i$, $1\leq i\leq \ell_1$,  and $\psi_j$,  $1\leq j\leq \ell_2$, 
of $h$ and $g$, respectively. It follows from Lemma \ref{gen1} that $f$ is minimal, and that the smallest principal curvature angles
$\varphi_1$ and $\psi_1$ of $h$ and $g$, respectively, satisfy \eqref{smallest}. Since $\mathcal{H}=0$, then \eqref{eq:r1r2b} holds.
If the principal curvatures of $h$ and $g$ have common multiplicities, then condition $(i)$ is satisfied by Lemma \ref{single}.
Otherwise, if $\ell_1$ or $\ell_2$, say, the former, is equal to $2$ or $4$, and the principal curvatures of $h$ do not have the same multiplicity,
then both $\ell_1, \ell_2\in \{2,4\}$ by Lemma \ref{l1=2}, and if the principal curvatures of $h$ and $g$  have  
 pairs of distinct multiplicities $(m_1, m_2)$ and $(p_1, p_2)$, then condition $(ii)$ holds by Lemma \ref{distinct}.
\end{proof}

\begin{corollary} There are no hypersurfaces $f\colon M^n\to \mathbb{S}_{c_1}^{k}\times \mathbb{S}_{c_2}^{n-k+1}$
in class~$\mathcal{A}$ with constant mean curvature and constant product angle function $\Theta\in (-1,1)$ 
if ${c_1}/{c_2}\in \mathbb{R}\setminus \mathbb{Q}$.
\end{corollary}

We now give several examples of minimal hypersurfaces in class~$\mathcal{A}$ 
with a constant product angle function $\Theta\in (-1,1)$ in certain products of two spheres of unit radii.

\begin{example} \label{ex:1sph} Let $h\colon M_1^{k-1}\to \mathbb{S}^{k}$ be an isoparametric hypersurface with $\ell\in \{2, 3, 4, 6\}$ distinct principal curvatures $\cot \varphi_1, \ldots, \cot \varphi_\ell$ of common multiplicity 
\begin{equation} \label{mult}
m_\ell\in \left\lbrace \begin{array}{l} \{m\}, \,\, \mbox{if} \,\, \ell=2\\
\{1,2,4,8\},\,\, \mbox{if} \,\, \ell=3\\
\{1,2\},\,\, \mbox{if} \,\, \ell=4\,\, \mbox{or}\,\, \ell=6,
\end{array}\right.
\end{equation}
so that $k-1=\ell m_\ell$. Let  $g\colon M_2^{\ell (k-1)}\to \mathbb{S}^{\ell (k-1)+1}$ be a geodesic hypersphere with principal curvature $\cot \psi_1$, with $\psi_1=\ell\varphi_1$ or $\psi_1=\pi-\ell\varphi_1$. Set $\theta=\arctan \ell$ in the former case and  $\theta=-\arctan \ell$ in the latter, and let $\delta > 0$ be such that
$h_{s \cos \theta}$ and $
g_{s \sin \theta} $
are immersions for all $s \in (-\delta, \delta)$. Then the map
\(
f \colon (-\delta, \delta) \times M_1^{k-1} \times M_2^{\ell (k-1)} \to  \mathbb{S}^{k} \times \mathbb{S}^{\ell (k-1)+1},
\)
defined by \eqref{definition of f} with $\theta$ as above, is a minimal hypersurface in class~$\mathcal{A}$ with a constant product angle function $\Theta={(\ell^2-1)}/{(\ell^2+1)}$.
\end{example}
\vspace{1ex}

\begin{example}  \label{ex:2sph} Let $h\colon M_1^{12}\to \mathbb{S}^{13}$ (respectively, $h\colon M_1^{24}\to \mathbb{S}^{25}$) be an isoparametric hypersurface with $3$  distinct principal curvatures of common multiplicity~$4$ (respectively, $8$), and 
let  $g\colon  \mathbb{S}_{\sec^2 t}^9\times  \mathbb{S}_{\csc^2 t}^9\to  \mathbb{S}^{19}$ (respectively, $g\colon  \mathbb{S}_{\sec^2 t}^{18}\times  \mathbb{S}_{\csc^2 t}^{18}\to  \mathbb{S}^{37}$) be an isoparametric hypersurface with $2$  distinct principal curvatures.  Assume that the  smallest principal curvature angles $\varphi_1$ and $\psi_1=t$ of $h$ and $g$, respectively, satisfy
$2\psi_1=3\varphi_1$ or $2\psi_1=\pi-3\varphi_1$. Set $\theta=\arctan \frac{3}{2}$ in the former case and  $\theta=-\arctan \frac{3}{2}$ in the latter, and let $\delta > 0$ be such that
$h_{s \cos \theta}$ and $
g_{s \sin \theta} $
are immersions for all $s \in (-\delta, \delta)$.  Then the map
$
f \colon (-\delta, \delta) \times M_1^{12} \times \mathbb{S}_{\sec^2 t}^9\times  \mathbb{S}_{\csc^2 t}^9 \to  \mathbb{S}^{13} \times \mathbb{S}^{19}
$ (respectively, $f \colon (-\delta, \delta) \times M_1^{24} \times \mathbb{S}_{\sec^2 t}^{18}\times  \mathbb{S}_{\csc^2 t}^{18} \to  \mathbb{S}^{25} \times \mathbb{S}^{37}$)
defined by \eqref{definition of f} with $\theta$ as above, is a minimal hypersurface in class~$\mathcal{A}$ with a constant product angle function $\Theta={5}/{13}$.
\end{example}
\vspace{1ex}

\begin{example}\label{ex:3sph}  Let $h\colon M_1^{4}\to \mathbb{S}^{5}$ (respectively, $h\colon M_1^{8}\to \mathbb{S}^{9}$) be an isoparametric hypersurface with $4$  distinct principal curvatures of common multiplicity~$1$ (respectively, $2$), and 
let  $g\colon  \mathbb{S}_{\sec^2 t}^4\times  \mathbb{S}_{\csc^2 t}^4\to  \mathbb{S}^{9}$ (respectively, $g\colon  \mathbb{S}_{\sec^2 t}^{8}\times  \mathbb{S}_{\csc^2 t}^{8}\to  \mathbb{S}^{17}$) be an isoparametric hypersurface with $2$  distinct principal curvatures. Assume that the  smallest principal curvature angles $\varphi_1$ and $\psi_1=~t$ of $h$ and $g$, respectively, satisfy
$\psi_1=2\varphi_1$ or $\psi_1=\pi/2-2\varphi_1$. Set $\theta=\arctan 2$ in the former case and  $\theta=-\arctan 2$ in the latter, and let $\delta > 0$ be such that
$h_{s \cos \theta}$ and $
g_{s \sin \theta} $
are immersions for all $s \in (-\delta, \delta)$.   Then the map
$
f \colon (-\delta, \delta) \times M_1^{4} \times \mathbb{S}_{\sec^2 t}^4\times  \mathbb{S}_{\csc^2 t}^4 \to  \mathbb{S}^{5} \times \mathbb{S}^{9}
$ (respectively, $f \colon (-\delta, \delta) \times M_1^{8} \times \mathbb{S}_{\sec^2 t}^{8}\times  \mathbb{S}_{\csc^2 t}^{8} \to  \mathbb{S}^{9} \times \mathbb{S}^{17}$),
defined by \eqref{definition of f} with $\theta$ as above, is a minimal hypersurface in class~$\mathcal{A}$ with a constant product angle function 
$\Theta={3}/{5}$.
\end{example}
\vspace{1ex}

\begin{example} \label{ex:4sph}  Let $h\colon M_1^{6}\to \mathbb{S}^{7}$ (respectively, $h\colon M_1^{12}\to \mathbb{S}^{13}$) be an isoparametric hypersurface with $6$  distinct principal curvatures of common multiplicity~$1$ (respectively, $2$), and 
let  $g\colon  \mathbb{S}_{\sec^2 t}^9\times  \mathbb{S}_{\csc^2 t}^9\to  \mathbb{S}^{19}$ (respectively, $g\colon  \mathbb{S}_{\sec^2 t}^{18}\times  \mathbb{S}_{\csc^2 t}^{18}\to  \mathbb{S}^{37}$) be an isoparametric hypersurface with $2$  distinct principal curvatures.   Assume that the  smallest principal curvature angles $\varphi_1$ and $\psi_1=t$ of $h$ and $g$, respectively, satisfy
$\psi_1=3\varphi_1$ or $\psi_1=\pi/2-3\varphi_1$. Set $\theta=\arctan 3$ in the former case and  $\theta=-\arctan 3$ in the latter, and let $\delta > 0$ be such that
$h_{s \cos \theta}$ and $
g_{s \sin \theta} $
are immersions for all $s \in (-\delta, \delta)$.  Then the map
$
f \colon (-\delta, \delta) \times M_1^{6} \times \mathbb{S}_{\sec^2 t}^9\times  \mathbb{S}_{\csc^2 t}^9 \to  \mathbb{S}^{7} \times \mathbb{S}^{19}
$ (respectively, $f \colon (-\delta, \delta) \times M_1^{12} \times \mathbb{S}_{\sec^2 t}^{18}\times  \mathbb{S}_{\csc^2 t}^{18} \to  \mathbb{S}^{19} \times \mathbb{S}^{37}$),
defined by \eqref{definition of f} with $\theta$ as above, is a minimal hypersurface in class~$\mathcal{A}$ with a constant product angle function $\Theta={4}/{5}$.
\end{example}
\vspace{1ex}

\begin{example} \label{ex:5sph}  Let $h\colon M_1^{6}\to \mathbb{S}^{7}$ (respectively, $h\colon M_1^{12}\to \mathbb{S}^{13}$) be an isoparametric hypersurface with $6$  distinct principal curvatures of common multiplicity~$1$ (respectively, $2$), and 
let  $g\colon  M_2^{12}\to  \mathbb{S}^{13}$ (respectively, $g\colon  M_2^{24}\to  \mathbb{S}^{25}$) be an isoparametric hypersurface with $3$  distinct principal curvatures of common multiplicity~$4$ (respectively, $8$).  Assume that the smallest principal curvature angles $\varphi_1$ and $\psi_1$ of $h$ and $g$, respectively, satisfy $\psi_1=2\varphi_1$ or $\psi_1=\pi/3-2\varphi_1$. Set $\theta=\arctan 2$ in the former case and  $\theta=-\arctan 2$ in the latter, and let $\delta > 0$ be such that
$h_{s \cos \theta}$ and $
g_{s \sin \theta} $
are immersions for all $s \in (-\delta, \delta)$.  Then the map
$
f \colon (-\delta, \delta) \times M_1^{6} \times M_2^{12} \to  \mathbb{S}^{7} \times \mathbb{S}^{13}
$ (respectively, $f \colon (-\delta, \delta) \times M_1^{12} \times  M_2^{24} \to  \mathbb{S}^{13} \times \mathbb{S}^{25}$),
defined by \eqref{definition of f} with $\theta$ as above, is a minimal hypersurface in class~$\mathcal{A}$ with constant product angle function $\Theta=\frac{3}{5}$.
\end{example}

\begin{rem} \label{remsph} Let $f \colon (-\delta, \delta) \times M^{m} \times N^{n} \to \mathbb{S}^{m+1} \times \mathbb{S}^{n+1}$, $m, n\geq 1$, be a minimal hypersurface with constant product angle function $\Theta=-\cos 2\theta$,  $\theta \in (0, 2\pi) \setminus \left\{ \tfrac{\pi}{2}, \pi, \tfrac{3\pi}{2} \right\}$, given by \eqref{definition of f} in terms of isoparametric hypersurfaces 
$h \colon M^{m} \to \mathbb{S}^{m+1}$ and $g \colon N^{n} \to \mathbb{S}^{n+1}$, all of whose principal curvatures have the same multiplicity. Let  $\tilde f \colon (-\delta, \delta) \times N^{n} \times M^{m} \to \mathbb{S}^{n+1} \times \mathbb{S}^{m+1}$ be given by \eqref{definition of f} with the roles of $h$ and $g$ interchanged, and in terms of an angle $\tilde{\theta}$ such that $\cot\tilde{\theta}=\tan \theta$. Then $\tilde f$ is also minimal and has a constant product angle function $\tilde \Theta=\cos 2\theta$.    
\end{rem}

The next examples of minimal hypersurfaces with a constant product angle function of certain products of two spheres
are constructed in terms of isoparametric hypersurfaces with $2$ or $4$ distinct principal curvatures with 
pairs of distinct multiplicities satisfying the condition $(ii)$ in Theorem \ref{prodspheres}.

\begin{example}\label{ex:6sph}   Let $h\colon M_1^{n}\to \mathbb{S}^{n+1}$  and $g\colon M_2^{n}\to \mathbb{S}^{n+1}$ be isoparametric hypersurfaces with $2$  distinct principal curvatures and pairs of distinct multiplicities $(m_1, m_2)$ and $(m_2, m_1)$, respectively.
Assume that the smallest principal curvature angles $\varphi_1$ and $\psi_1$ of $h$ and $g$, respectively, satisfy $\psi_1=\pi/2-\varphi_1$ and that $h_{s}$ and $g_{s}$ are immersions for all $s \in (-\sqrt{2}\,\delta, \sqrt{2}\,\delta)$.
  Then the map
\(
f \colon (-\delta, \delta)\times  \mathbb{S}_{\sec^2 t}^{m_1}\times  \mathbb{S}_{\csc^2 t}^{m_2} \times \mathbb{S}_{\csc^2 t}^{m_2}\times  \mathbb{S}_{\sec^2 t}^{m_1}
 \to  \mathbb{S}^{n+1} \times \mathbb{S}^{n+1},
 \)
defined by \eqref{definition of f} with $\theta=3\pi/4$, is a minimal hypersurface in class~$\mathcal{A}$ with a vanishing product angle function.
\end{example}

\begin{example}\label{ex:8sph}  Let $h\colon M_1^{n}\to \mathbb{S}^{n+1}$  and $g\colon M_2^{n}\to \mathbb{S}^{n+1}$ be isoparametric hypersurfaces with $4$  distinct principal curvatures and pairs of distinct multiplicities $(m_1, m_2)$ and $(m_2, m_1)$, respectively. Assume that the smallest principal curvature angles $\varphi_1$ and $\psi_1$ of $h$ and $g$, respectively, satisfy $\psi_1=\pi/4-\varphi_1$.  Let $\delta > 0$ be such that
$h_{s}$ and $
g_{s} $
are immersions for all $s \in (-\sqrt{2}\,\delta, \sqrt{2}\,\delta)$.  Then the map
$
f \colon (-\delta, \delta) \times M_1^{n} \times M_2^n\to  \mathbb{S}^{n+1} \times \mathbb{S}^{n+1},
$ 
defined by \eqref{definition of f} with $\theta=3\pi/4$,  is a minimal hypersurface in class~$\mathcal{A}$  with a vanishing product angle function.
\end{example}

\begin{example}\label{ex:7sph}  Let $h\colon M_1^{n}\to \mathbb{S}^{n+1}$  be an isoparametric hypersurface with $4$  distinct principal curvatures and a pair of distinct multiplicities $(m_1, m_2)$. Let  $g\colon  \mathbb{S}_{\sec^2 t}^{4m_1}\times  \mathbb{S}_{\csc^2 t}^{4m_2}\to  \mathbb{S}^{2n+1}$ 
(respectively, $g\colon  \mathbb{S}_{\sec^2 t}^{4m_2}\times  \mathbb{S}_{\csc^2 t}^{4m_1}\to  \mathbb{S}^{2n+1}$) be an isoparametric hypersurface with $2$  distinct principal curvatures.  Assume that the smallest principal curvature angles $\varphi_1$ and $\psi_1=t$ of $h$ and $g$, respectively, satisfy $2\varphi_1=\psi_1$ (respectively,  $2\varphi_1=\pi/2-\psi_1$).  Let $\theta=\arctan 2$ (respectively, $\theta=\arctan (-2)$), and let $\delta > 0$ be such that
$h_{s\cos \theta}$ and $
g_{s\sin\theta} $
are immersions for all $s \in (-\delta, \delta)$.  Then the map
\(
f \colon (-\delta, \delta) \times M_1^{n} \times  \mathbb{S}_{\sec^2 t}^{4m_1}\times  \mathbb{S}_{\csc^2 t}^{4m_2}\to  \mathbb{S}^{n+1} \times \mathbb{S}^{2n+1},
\)
(respectively,
\(
f \colon (-\delta, \delta) \times M_1^{n} \times  \mathbb{S}_{\sec^2 t}^{4m_2}\times  \mathbb{S}_{\csc^2 t}^{4m_1}
\to   \mathbb{S}^{n+1} \times \mathbb{S}^{2n+1}),
\)
defined by \eqref{definition of f} with $\theta$ as above,  is a minimal hypersurface in class~$\mathcal{A}$  with a constant product angle function $\Theta=3/5$.
\end{example}

\begin{example}\label{ex:9sph}  
Let $h\colon  \mathbb{S}_{\sec^2 t}^{4m_1}\times  \mathbb{S}_{\csc^2 t}^{4m_2}\to  \mathbb{S}^{2n+1}$ 
(respectively, $h\colon  \mathbb{S}_{\sec^2 t}^{4m_2}\times  \mathbb{S}_{\csc^2 t}^{4m_1}\to  \mathbb{S}^{2n+1}$) be an isoparametric hypersurface with $2$  distinct principal curvatures and let $g\colon M_2^{n}\to \mathbb{S}^{n+1}$  be an isoparametric hypersurface with $4$  distinct principal curvatures and a pair of distinct multiplicities $(m_1, m_2)$. Assume that the smallest principal curvature angles $\varphi_1=t$ and $\psi_1$ of $h$ and $g$, respectively, satisfy $\varphi_1=2\psi_1$ (respectively,  $\varphi_1=\pi/2-2\psi_1$).  Let $\theta=\arctan (1/2)$ (respectively, $\theta=\arctan (-1/2)$), and let $\delta > 0$ be such that
$h_{s\cos \theta}$ and $
g_{s\sin\theta} $
are immersions for all $s \in (-\delta, \delta)$.  Then the map
\(
f \colon (-\delta, \delta) \times \mathbb{S}_{\sec^2 t}^{4m_1}\times  \mathbb{S}_{\csc^2 t}^{4m_2} \times M_2^{n}\to   \mathbb{S}^{2n+1}\times \mathbb{S}^{n+1},
\)
(respectively,
\(f \colon (-\delta, \delta) \times \mathbb{S}_{\sec^2 t}^{4m_2}\times \mathbb{S}_{\csc^2 t}^{4m_1} \times M_2^{n}\to \mathbb{S}^{2n+1} \times \mathbb{S}^{n+1}\)),
defined by \eqref{definition of f} with $\theta$ as above,  is a minimal hypersurface in class~$\mathcal{A}$  with a constant product angle function $\Theta=-3/5$.
\end{example}
\begin{rem} The possible pairs $(m_1, m_2)$ of distinct multiplicities of isoparametric hypersurfaces with $4$ distinct principal curvatures are $(4, 5)$, and 
the pairs $(m, k\delta(m)-1)$ that appear in the $FKM$-class of isoparametric hypersurfaces constructed by Ferus, Karcher and M\"unzner \cite{FKM} from representations of Clifford algebras, where $\delta(m)$ is the dimension of an irreducible module of the Clifford algebra $C_{m-1}$. In particular, $(1,k)$,  $(3,4k)$ and  $(7,8k)$, $k\in \mathbb{N}$, are such pairs. If $(m_1, m_2)$ arises as the pair of distinct multiplicities of an isoparametric hypersurface with $4$ distinct principal curvatures, then  
$(m_2, m_1)$ also arises as the corresponding pair of a suitable parallel isoparametric hypersurface. However, there exist noncongruent isoparametric hypersurfaces
 in the $FKM$-class with the same pair  $(m_1, m_2)$ of distinct multiplicities.    
\end{rem}

\begin{corollary}\label{spheres}
Any   hypersurface $f \colon M^{n} \to \mathbb{S}^{k} \times \mathbb{S}^{n-k+1}$, 
$2 \leq k \leq n-1$, in class~$\mathcal{A}$ with constant mean curvature and a constant product angle function $\Theta\in (-1, 1)$
 is minimal and locally given as in Examples \ref{example1} to \ref{ex:9sph}, or it is locally given as in Remark \ref{remsph} in terms of the hypersurfaces in one of the Examples~\ref{ex:1sph} to \ref{ex:5sph}. 
\end{corollary}

\begin{proof}
By Theorem \ref{prodspheres}, there are isoparametric hypersurfaces 
$h \colon M_1^{k-1} \to \mathbb{S}^{k}$ and $g \colon M_2^{n-k} \to  \mathbb{S}^{n-k+1}$ that satisfy  \eqref{smallest} and conditions $(i)$ or $(ii)$, 
 and $\theta \in (0, 2\pi) \setminus \left\{ \tfrac{\pi}{2}, \pi, \tfrac{3\pi}{2} \right\}$ given by \eqref{tan}, such that (locally)
$ M^n = (-\delta, \delta) \times M_1^{k-1} \times M_2^{n-k}$ and  $f$ is given by \eqref{definition of f}. 
 
 Assume first that condition $(i)$ is satisfied and let $\ell_1$ and $\ell_2$ be the numbers of distinct principal curvatures of $h$ and $g$, respectively, 
 which have common multiplicities $m$ and $p$, respectively. In the following, we will determine for which pairs $(\ell_1,m)$ and $(\ell_2,p)$ the equality  
 \eqref{relb} holds. By Remark~\ref{remsph}, we can assume that $\ell_2\leq \ell_1$.
  
    First notice that, if $\ell_2= \ell_1$, then $m=p$ by \eqref{relb}, hence $k-1=m\ell_1=p\ell_2=n-k$. Then \eqref{tan} gives $\tan \theta=1$ if $\varphi_1=\psi_1$, and  $\tan \theta=-1$ if $\varphi_1=\pi-\psi_1$.
    Therefore, $h=\Phi\circ g$ for some isometry $\Phi\colon \mathbb{S}^{k}\to \mathbb{S}^{k}$, with $\Phi_* N^h = N^g$ in the former case and $\Phi_* N^h = -N^g$ in the latter, and $f$ has a vanishing product angle function. Thus, $f$ is as in Example \ref{example1}. Therefore, we may assume from now on that $\ell_2< \ell_1$. 

  Suppose first that $\ell_2=1$, so that \eqref{relb} reduces to $m\ell_1^2=p$.   If $\ell_1=\ell\in \{2,3,4,6\}$, then \eqref{mult} gives the corresponding possible multiplicities for $m=m_\ell$. We have $n-k=p=m\ell^2=\ell (k-1)$, while $\tan \theta=\ell$
  if $\psi_1=\ell \varphi_1$,  and $\tan \theta=-\ell$
  if $\psi_1=\pi-\ell \varphi_1$  by \eqref{tan}. Thus, $f$ is as in Example \ref{ex:1sph}.

  Now assume that  $\ell_2=2$, so  \eqref{relb} gives $m\ell_1^2=4p$. 
  If $\ell_1=3$, then $9m=4p$. Since $9m\in \{9, 18, 36, 72\}$, then $m=4$ or $m=8$. Thus, $k-1=m\ell_1$ is either $12$ or $24$ and $n-k$ is either $18$ or $36$, respectively. By \eqref{tan}, we have $\tan \theta=3/2$ if $2\psi_1=3\varphi_1$, and $\tan \theta=-3/2$ if $2\psi_1=\pi-3\varphi_1$. Thus, $f$ is as in Example \ref{ex:2sph}.
  If $\ell_1=4$, then $16m=4p$, and hence $p=4m$. Since $m\in \{1,2\}$, $k-1=m\ell_1$ is $4$ or $8$ and $n-k=2p$ is either $8$ or $16$, respectively. By \eqref{tan},  $\tan \theta=2$ if $\psi_1=2\varphi_1$, and $\tan \theta=2$ if $\psi_1=\pi/2-2\varphi_1$. Thus, $f$ is as in Example \ref{ex:3sph}. If $\ell_1=6$, then $36m=4p$, and hence $p=9m$. Since $m\in \{1,2\}$,  $k-1=m\ell_1$ is equal to $6$ or $12$, and $n-k=2p$ is $18$ or $36$, respectively. Eq. \eqref{tan} gives  $\tan \theta=3$ if $\psi_1=3\varphi_1$, and  $\tan \theta=-3$ if $\psi_1=\pi/2-3\varphi_1$. Thus, $f$ is as in Example \ref{ex:4sph}.

  If  $\ell_2=3$, then  $m\ell_1^2=9p$ by \eqref{relb}. For $\ell_1=4$, we have $16m=9p$. Since $m\in \{1, 2\}$ and $p\in \{1, 2, 4, 8\}$,
  $16m\in \{16, 32\}$ and $9p\in \{9, 18, 36, 72\}$, this case does not occur. If $\ell_1=6$, then $36m=9p$, and hence $p=4m$. Thus $p\in \{4,8\}$, hence  $k-1=m\ell_1$ is $6$ or $12$ and $n-k=3p$ is $12$ or $24$, respectively. By \eqref{tan}, we have $\tan \theta=2$ if $\psi_1=2\varphi_1$, and $\tan \theta=-2$ if $\psi_1=\pi/3-2\varphi_1$, respectively. Thus, $f$ is as in Example~\ref{ex:5sph}. 

  If $\ell_2=4$ and $\ell_1=6$, then $36m=16p$, and hence $9m=4p$. Since $m,p \in \{1, 2\}$,  this case does not occur.

  Suppose now that condition $(ii)$ is satisfied. If $\ell_1=\ell_2=2$ and $\varphi_1=\psi_1$, then $(m_1, m_2)=(p_1, p_2)$ by $(a)$
, and $\tan \theta=1$ by \eqref{tan}. Hence $h$ and $g$ are congruent, and $f$ is as in Example \ref{example1}. If $\ell_1=\ell_2=2$ and $\varphi_1=\pi/2-\psi_1$, then $(m_1, m_2)=(p_2, p_1)$ by $(b)$, and $\tan \theta=-1$ by \eqref{tan}. Hence  $f$ is as in Example \ref{ex:6sph}.

The case  $\ell_1=\ell_2=4$ is similar: if $\varphi_1=\psi_1$, then $(m_1, m_2)=(p_1, p_2)$ by $(a)$, and $\tan \theta=1$ by \eqref{tan}. Hence $h$ and $g$ are congruent, and $f$ is as in Example \ref{example1}. If $\varphi_1=\pi/4-\psi_1$, then $(m_1, m_2)=(p_2, p_1)$ by $(b)$, and $\tan \theta=-1$ by \eqref{tan}. Hence  $f$ is as in Example \ref{ex:8sph}.

Assume now that  $\ell_1=4$ and $\ell_2=2$.  If $\psi_1=2\varphi_1$ (respectively, $\psi_1=\pi/2-2\varphi_1$),  by $(a)$ we have $4(m_1, m_2)=(p_1, p_2)$  (respectively, 
$4(m_1, m_2)=(p_2, p_1)$), and $\tan \theta=2$   (respectively, $\tan \theta=-2$) by \eqref{tan}. Hence  $f$ is as in Example \ref{ex:7sph}. The case  $\ell_1=2$ and $\ell_2=4$ is similar and leads to Example \ref{ex:9sph}.
\end{proof}

\subsubsection{The case $c_1,c_2<0$} 
We now address the case where $c_1$ and $c_2$ are negative. First we give examples of hypersurfaces in $\mathbb{H}^k_{c_1}\times\mathbb{H}^{n-k+1}_{c_2}$ in class~$\mathcal{A}$ with constant mean curvature and a constant product angle function $\Theta\in(-1,1)$.
From now on, $r_i>0$ is defined by $c_i = -1/r_i^2$, $i=1,2$.

\begin{example}\label{ex:horospheres} Let 
\(h\colon \mathbb{R}^{k-1}\to\mathbb{H}_{c_1}^{k}\) 
and 
\(g\colon \mathbb{R}^{n-k}\to\mathbb{H}_{c_2}^{n-k+1}\), $2\leq k\leq n-~1$, 
be horospheres with principal curvatures $\lambda=\pm{1}/{r_1}$ and $\mu=\pm{1}/{r_2}$,
respectively. By Corollary \ref{cor: PAC}, 
$
f\colon \mathbb{R}\times \mathbb{R}^{k-1}\times \mathbb{R}^{n-k}\to\mathbb{H}_{c_1}^{k}\times\mathbb{H}_{c_2}^{n-k+1}
$,
given by \eqref{definition of f} with $
\theta\in(0,2\pi)\setminus\left\{{\pi}/{2},\pi,{3\pi}/{2}\right\}
$,
 is a hypersurface in class $\mathcal{A}$ with a constant product angle function $\Theta=-\cos 2\theta$.
 By \eqref{eq: constant theta},  it has constant mean curvature given by 
 \(n\mathcal{H}=-(k-1)\lambda\sin\theta+(n-k)\mu\cos\theta\).
\end{example}

\begin{example}\label{ex:case-i}
Let $h\colon M_1^{k-1}\to\mathbb{H}_{c_1}^{k}$ and $g\colon M_2^{n-k}\to\mathbb{H}_{c_2}^{n-k+1}$, $2\leq k\leq n-1$, be geodesic hyperspheres with principal curvatures
$
\lambda=r_1^{-1}\coth\left(\varphi/r_1\right)
$ 
and
$\mu={r_2}^{-1}\coth\left({\psi}/{r_2}\right)$,
respectively.  Assume that $(n-k)c_2=(k-1)c_1$, and that $\psi=\varphi\tan\theta$ for $\theta\in(0,2\pi)\setminus\{\pi/2,\pi,3\pi/2\}$  such that $|\tan\theta|=r_2/r_1$. Let  $\delta>0$  be such that 
$h_{s\cos\theta}$ and $g_{s\sin\theta}$ are immersions for all  $s\in (-\delta, \delta)$. By  Corollary~\ref{cor: PAC} and \eqref{eq: constant theta},
the map 
$
f\colon M^{n}=(-\delta, \delta)\times M_1^{k-1}\times M_2^{n-k}\to\mathbb{H}_{c_1}^{k}\times\mathbb{H}_{c_2}^{n-k+1},
$
given by \eqref{definition of f} with $\theta$ as above, defines a minimal hypersuface in class $\mathcal{A}$ with a constant 
 product angle function 
$\Theta=(c_1-c_2)/(c_1+c_2)$.
\end{example}

\begin{example}\label{ex:case-ii} Let 
$
h\colon M_1^{2m}\to\mathbb{H}_{c_1}^{2m+1}
$, $m\geq 1$, 
be an isoparametric hypersurface with two distinct principal curvatures
$
\lambda_1={r_1}^{-1}\tanh\left({\varphi}/{r_1}\right)
$
and
$
\lambda_2={r_1}^{-1}\coth\left({\varphi}/{r_1}\right),
$
both of multiplicity $m$, and let $g\colon \mathbb{S}_\sigma^{n}\to\mathbb{H}_{c_2}^{n+1}$ be a geodesic hypersphere with a principal curvature
$
\mu={r_2}^{-1}\coth\left({\psi}/{r_2}\right),
$
so $\sigma=-c_2\operatorname{csch}^2(\psi/r_2)$. Assume that $c_2n=4mc_1$ and that  $\psi=\varphi\tan\theta$ for $\theta\in(0,2\pi)\setminus\{\pi/2,\pi,3\pi/2\}$ such that $|\tan\theta|=2r_2/r_1$. Let $\delta>0$ be such that $h_{s\cos\theta}$ and $g_{s\sin\theta}$ are immersions for all $s\in (-\delta, \delta)$. 
By  Corollary \ref{cor: PAC} and \eqref{eq: constant theta}, the map 
$
f\colon (-\delta, \delta)\times M_1^{2m}\times \mathbb{S}_\sigma^{n}\to\mathbb{H}_{c_1}^{2m+1}\times\mathbb{H}_{c_2}^{n+1},
$
given by \eqref{definition of f} with $\theta$ as above, defines  a minimal hypersuface in class $\mathcal{A}$ with a constant 
 product angle function 
$\Theta=(4c_1-c_2)/(4c_1+c_2)$.
\end{example}

\begin{example}\label{ex:case-iv}
Let $h\colon \mathbb{S}^{k-1}_\sigma\to\mathbb{H}_{c_1}^{k}$, $k\geq 2$, $\sigma=-c_1\operatorname{csch}^2(\varphi/r_1)$, be a geodesic sphere with a principal curvature
$
\lambda={r_1}^{-1}\coth\left({\varphi}/{r_1}\right)
$
and let  $g\colon M_2^{n-k}\to\mathbb{H}_{c_2}^{n-k+1}$ be an isoparametric hypersurface with two distinct principal curvatures
$
\mu_1={r_2}^{-1}\tanh\left({\psi}/{r_2}\right)
$ 
 and 
$
\mu_2={r_2}^{-1}\coth\left({\psi}/{r_2}\right),
$
both of multiplicity $p$. Assume that $c_1(k-1)=2c_2(n-k)$ and that  $\psi=\varphi\tan\theta$ for $\theta\in(0,2\pi)\setminus\{\pi/2,\pi,3\pi/2\}$ such that $|\tan\theta|=r_2/(2r_1)$.
Let $\delta>0$ be such that $h_{s\cos\theta}$ and $g_{s\sin\theta}$ are immersions for all
 $s\in (-\delta,\delta).$ By Corollary \ref{cor: PAC} and \eqref{eq: constant theta}, 
the map 
$
f\colon M^{n}= (-\delta,\delta)\times \mathbb{S}^{k-1}_\sigma\times M_2^{n-k}\to\mathbb{H}_{c_1}^{k}\times\mathbb{H}_{c_2}^{n-k+1},
$
given by \eqref{definition of f} with $\theta$ as above, defines  a minimal hypersurface in class $\mathcal{A}$ with a constant product angle function
$
\Theta=(c_1-4c_2)/(c_1+4c_2).
$
\end{example}

\begin{example}\label{ex:case-v}
Let $h\colon M_1^{k-1}\to \mathbb{H}^k_{c_1}$  and  $g\colon M_2^{n-k}\to \mathbb{H}^{n-k+1}_{c_2}$, $3\leq k\leq n-2$, be isoparametric hypersurfaces with two distinct principal curvatures $\lambda_1=r_1^{-1}\tanh(\varphi/r_1)$, $\lambda_2=r_1^{-1}\coth(\varphi/r_1)$, of multiplicities $1\leq m\leq  k-2$ and $k-1-m$, and $\mu_1=r_2^{-1}\tanh(\psi/r_2)$, $\mu_2=r_2^{-1}\coth(\psi/r_2)$, of multiplicities $p$ and $n-k-p$, respectively. Assume that $pc_2=mc_1$, $(n-k)c_2=(k-1)c_1$ and that  $\psi = \varphi\tan \theta $ for  $\theta\in(0,2\pi)\setminus\{\pi/2,\pi,3\pi/2\}$ such that $|\tan\theta|=r_2/r_1$. Let $\delta>0$ be such that $h_{s\cos\theta}$ and $g_{s\sin\theta}$ are immersions for all
 $s\in (-\delta,\delta)$. By Corollary \ref{cor: PAC} and \eqref{eq: constant theta}, 
the map $f\colon  (-\delta,\delta)\times M_1^{k-1}\times  M_2^{n-k}\to  \mathbb{H}^k_{c_1}\times \mathbb{H}^{n-k+1}_{c_2}$, given by \eqref{definition of f} with $\theta$ as above, defines  a minimal hypersurface in class $\mathcal{A}$ with a constant product angle function
 $\Theta=(c_1-c_2)/(c_1+c_2)$.
\end{example}

\begin{example}\label{ex:case-ix}
Let $h\colon \mathbb{H}_{\sigma_1}^{k-1}\to\mathbb{H}_{c_1}^{k}$, $\sigma_1=c_1\operatorname{sech}^2(\varphi/r_1)$,  and $g\colon \mathbb{H}_{\sigma_2}^{n-k}\to\mathbb{H}_{c_2}^{n-k+1}$, $\sigma_2=c_2\operatorname{sech}^2(\psi/r_2)$,  be equidistant or totally geodesic hypersurfaces with principal curvatures
$
\lambda={r_1}^{-1}\tanh\left({\varphi}/{r_1}\right)
$
 and
$
\mu={r_2}^{-1}\tanh\left({\psi}/{r_2}\right),
$
respectively. Assume that $c_2(n-k)=c_1(k-1)$ and that  $\psi=\varphi\tan\theta$ for $\theta\in(0,2\pi)\setminus\{\pi/2,\pi,3\pi/2\}$ such that 
$|\tan\theta|=r_2/r_1$. By Corollary \ref{cor: PAC} and \eqref{eq: constant theta}, 
the map 
$
f\colon \mathbb{R}\times \mathbb{H}_{\sigma_1}^{k-1}\times \mathbb{H}_{\sigma_2}^{n-k}\to\mathbb{H}_{c_1}^{k}\times\mathbb{H}_{c_2}^{n-k+1},
$
given by \eqref{definition of f} with $\theta$ as above, defines a minimal hypersurface with a constant  product angle function
$
\Theta=(c_1-c_2)/(c_1+c_2).
$
\end{example}

The following result shows that the preceding examples exhaust all possible ones when $c_1$ and $c_2$ are both negative. 

\begin{proposition}\label{thm:cmcpacgeral}
Let
$
f \colon M^{n} \to \mathbb{H}_{c_1}^{k} \times \mathbb{H}_{c_2}^{n-k+1} $, $2 \leq k \leq n-1$, 
be a hypersurface in class~$\mathcal{A}$ with constant mean curvature $\mathcal{H}$ and a constant product angle function $\Theta\in (-1, 1)$. 
Then $f$ is locally as in one of Examples \ref{ex:horospheres} to \ref{ex:case-ix}.
\end{proposition}

\begin{proof}
By Corollary \ref{cor: PAC}, there exist isoparametric hypersurfaces
$h \colon M_1^{k-1} \to \mathbb{H}^{k}$ and
$g \colon M_2^{n-k} \to \mathbb{H}^{n-k+1}$, and a constant
$\theta \in (0,2\pi) \setminus
\left\{ \tfrac{\pi}{2}, \pi, \tfrac{3\pi}{2} \right\}$,
such that, locally,
$M^n = (-\delta, \delta) \times M_1^{k-1} \times M_2^{n-k}$ and
$f$ is given by \eqref{definition of f}.

Suppose first that both $h$ and $g$ are horospheres, with principal curvatures $\lambda =\pm 1/r_1$ and $\mu =\pm
1/r_2$, respectively. Then the mean curvatures $\mathcal{H}^h$ and $\mathcal{H}^g$ of $h_s$ and $g_s$, respectively, satisfy
$\mathcal{H}^h(s)=\lambda$ and $\mathcal{H}^g(s)=\mu$ for all $s\in(-\delta,\delta)$. Hence, \eqref{eq: constant theta 2} becomes
$$
-(k-1)\lambda\tan\theta
+(n-k)\mu
=n\mathcal{H}\sec\theta.
$$
Thus, $f$ is locally given as in Example~\ref{ex:horospheres}.

Suppose now that only one of $h$ and $g$, say $h$, is a horosphere. Then
$$
\mathcal{H}^g(s)
=\frac{1}{(n-k)r_2}\left(
p\tanh\left(\frac{\psi-s}{r_2} \right)
+(n-k-p)\coth \left(\frac{\psi-s}{r_2} \right)
\right),
$$
for all $s\in(-\delta,\delta)$, where $ 0\leq p\leq n-k$; the cases $p=0$ and $p=n-k$ correspond to $g$ being a geodesic hypersphere and an equidistant hypersurface, respectively. Thus, \eqref{eq: constant theta 2} becomes
$$
(n-k)\mathcal{H}^g(z\tan\theta)
=n\mathcal{H}\sec\theta
+(k-1)\lambda\tan\theta
$$
for every $z\in\mathbb{C}$ for which $\mathcal{H}^g(z\tan\theta)$ is defined. Since the left-hand side is not constant, this case cannot occur.

We now consider the case where neither \(h\) nor \(g\) is a horosphere. Set \(a=\tan\theta\). For \(0\leq p\leq n-k\) and \(0\leq m\leq k-1\), it follows from \eqref{eq: constant theta 2} that
\begin{equation}\label{newgeral}
\begin{array}{l}
\displaystyle
\frac{p}{r_2}\tanh\left(\frac{\psi-az}{r_2}\right)
+\frac{n-k-p}{r_2}\coth\left(\frac{\psi-az}{r_2}\right)
\\[1.2ex]
\displaystyle
=\ n\mathcal H\sec\theta
+\frac{ma}{r_1}\tanh\left(\frac{\varphi-z}{r_1}\right)
+\frac{(k-1-m)a}{r_1}\coth\left(\frac{\varphi-z}{r_1}\right),
\end{array}
\end{equation}
for all \(z\in\mathbb C\) for which both sides are defined, where \(\varphi\) and \(\psi\) are the respective smallest principal curvature  angles of \(h\) and \(g\).

We consider the following possible cases:
\begin{multicols}{2}
\begin{itemize}
\item[(i)] $p=0$ and $m=0$;
\item[(ii)] $p=0$ and $0<m<k-1$;
\item[(iii)] $p=0$ and $m=k-1$;
\item[(iv)] $0<p<n-k$ and $m=0$;
\item[(v)] $0<p<n-k$ and $0<m<k-1$;
\item[(vi)] $0<p<n-k$ and $m=k-1$;
\item[(vii)] $p=n-k$ and $m=0$;
\item[(viii)] $p=n-k$ and $0<m<k-1$;
\item[(ix)] $p=n-k$ and $m=k-1$.
\end{itemize}
\end{multicols}

In case $(i)$, the sets of poles of the two sides of \eqref{newgeral} are, respectively,
$$
P_1=\left\{\frac{\psi}{a}+ \frac{q\pi i r_2}{a}:q\in\mathbb Z\right\}
\qquad\text{and}\qquad
P_2=\left\{\varphi+q\pi i r_1:q\in\mathbb Z\right\}.
$$
Since the two sides of \eqref{newgeral} coincide, we have $P_1=P_2$. As before, this implies $
\frac{\pi i r_2}{a}\mathbb Z=\pi i r_1\mathbb Z,
$ and, consequently, $a=\pm r_2/r_1.$ Using again that $P_1=P_2$ and that $\varphi,\frac{\psi}{a}\in\mathbb R$, it follows that $\psi=a\varphi$.
Substituting $\psi=a\varphi$ and 
\[
\coth\left(\frac{a(\varphi-z)}{r_2}\right)
=\operatorname{sgn}(a)\coth\left(\frac{\varphi-z}{r_1}\right)
\]
into \eqref{newgeral} yields
\begin{equation}\label{newgi}
\frac{(n-k)\operatorname{sgn}(a)}{r_2}\coth\left(\frac{\varphi-z}{r_1}\right)
=n\mathcal{H}\sec\theta
+\frac{(k-1)a}{r_1}\coth\left(\frac{\varphi-z}{r_1}\right).
\end{equation}
Since the two sides of \eqref{newgi} are identical meromorphic functions of $z$,  the residues at their common simple poles
$
z_q=\varphi+q\pi i r_1, q\in\mathbb Z
$,
coincide. At each $z_q$, the function $\coth\bigl((\varphi-z)/r_1\bigr)$ has residue $-r_1$, while the constant term $n\mathcal{H}\sec\theta$ is holomorphic. Equating residues therefore gives
\[
\frac{(n-k)\operatorname{sgn}(a)}{r_2}(-r_1)
=
\frac{(k-1)a}{r_1}(-r_1),
\]
that is, $(n-k)c_2=(k-1)c_1$.
Therefore,  $f$ is locally as in Example~\ref{ex:case-i}.

The same pole-and-residue argument used in case $(i)$ applies verbatim to cases $(ii)$, $(iv)$, $(v)$ and $(ix)$, yielding Examples~\ref{ex:case-ii}, \ref{ex:case-iv}, \ref{ex:case-v} and \ref{ex:case-ix}, respectively.

To complete the proof, we claim that cases $(iii), (vi), (vii),$ and $(viii)$ cannot occur. We argue only for case 
$(iii)$; the remaining cases are analogous.
Suppose that $p=0$ and $m=k-1$. Then, the sets of poles of the two sides of~\eqref{newgeral} are
$$
P_1=\left\{\frac{\psi}{a}+\frac{q\pi i r_2}{a}:q\in\mathbb Z\right\}
\quad\text{and}\quad
P_2=\left\{\varphi+\left(q+\frac12\right)\pi i r_1:q\in\mathbb Z\right\}.
$$
Since the two sides coincide, we must have $P_1=P_2$. The equality
$P_1-P_1=P_2-P_2$ yields
$
\frac{\pi i r_2}{a}\mathbb Z=\pi i r_1\mathbb Z,
$
and hence $a=\pm r_2/r_1$. Consequently, the imaginary parts of the poles in $P_1$ belong to $\pi r_1\mathbb Z$, whereas those of the poles in $P_2$ belong to $\pi r_1/2+\pi r_1\mathbb Z$. This is a contradiction. Therefore, case $(iii)$ cannot occur.
\end{proof}

\subsubsection{The case $c_1>0$ and $c_2 \leq 0$}

In this subsection we prove the following.

\begin{proposition}\label{SxH_SxR}
There exists no hypersurface
$f\colon M^{n}\to\mathbb{S}^{k}_{c_1}\times\mathbb{Q}_{c_2}^{n-k+1}$, with  $c_2\le~0$ and
$2\le k\le n-1$, in class $\mathcal{A}$ with constant mean curvature and a
constant product angle function $\Theta\in(-1,1)$.
\end{proposition}

\begin{proof}
Assume otherwise. By Corollary~\ref{cor: PAC}, there exist
isoparametric hypersurfaces $h\colon M_1^{k-1}\to\mathbb{S}^{k}_{c_1}$ and
$g\colon M_2^{n-k}\to\mathbb{Q}_{c}^{n-k+1}$, and a constant
$\theta\in(0,2\pi)\setminus\{\tfrac{\pi}{2},\pi,\tfrac{3\pi}{2}\}$, such that
(locally) $M^n=(-\delta,\delta)\times M_1^{k-1}\times M_2^{n-k}$ and $f$ is
given by \eqref{definition of f}.

Write $c_1=1/r_1^2$ with $r_1>0$. Let $\varphi_1,\ldots,\varphi_{\ell_1}$ be the principal curvature angles of
$h$, with $\varphi_i=\varphi_1+{(i-1)\pi r_1}/{\ell_1}$,
$1\le i\le\ell_1$, and let $m_1,\ldots,m_{\ell_1}$ be their multiplicities.
Then  the mean curvature of the parallel hypersurface $h_s$ is 
\[
(k-1)\mathcal{H}^h(s)=\frac{1}{r_1}\sum_{i=1}^{\ell_1}m_i\,
\cot\!\left(\frac{\varphi_i-s}{r_1}\right).
\]

Suppose that $c_2=0$. Since  $\mathcal{H}^h(s\cos\theta)$ is not constant, 
it follows from \eqref{eq: constant theta 2} that $g$ cannot be totally geodesic.
If $g$ is a generalized cylinder, then
\[
(n-k)\mathcal{H}^g(s)=\frac{(n-k-\nu_g)\mu}{1-s\mu},
\]
where $\mu$ is the nonzero principal curvature of $g$ and $\nu_g$ its index
of relative nullity.
 Eq. \eqref{eq: constant theta 2} becomes
\begin{equation}\label{epsilon0gen}
\frac{(n-k-\nu_g)\mu}{1-z\mu\tan\theta}
=n\mathcal{H}\sec\theta + \frac{\tan\theta}{r_1}\sum_{i=1}^{\ell_1}m_i\,
\cot\!\left(\frac{\varphi_i-z}{r_1}\right)
\end{equation}
for any $z \in \mathbb{C}$ such that $z \neq \cot\theta/\mu$ and
$z \neq \varphi_1 + (p + i/\ell_1)\pi r_1$ for all $0 \le i \le \ell_1-1$
and $p \in \mathbb{Z}$.
Since the left-hand side of \eqref{epsilon0gen} has a single pole at
$z=\cot\theta/\mu$, while the right-hand side has infinitely many pairwise distinct simple poles
at $z=\varphi_i-p\pi r_1$, for $1\le i\le\ell_1$ and $p\in\mathbb{Z}$, all
with nonzero residues, this case is also not possible. Thus, no such hypersurface exists for $c_2=0$.
The argument for  $c_2<0$ is similar, using the expressions 
\eqref{hip1} and \eqref{hip2} for the mean curvature of the parallel hypersurface $g_s$.
\end{proof}

\subsubsection{The case $c_1<0$ and $c_2=0$.} There is no loss of generality in assuming $c_1=-1$. We start with a simple example, which we then prove to be the only possible one.

\begin{example}\label{ex: flat-horospherical} 
Let $h \colon \mathbb{R}^{k-1} \to \mathbb{H}^{k}$ be a horosphere with principal curvature $\lambda=\pm 1$, and
let $g \colon \mathbb{R}^{\,n-k} \to \mathbb{R}^{\,n-k+1}$ be a hyperplane. 
Then, the map $f \colon \mathbb{R} \times \mathbb{R}^{k-1} \times \mathbb{R}^{\,n-k} \longrightarrow \mathbb{H}^{k} \times \mathbb{R}^{\,n-k+1}$, 
given by \eqref{definition of f} for any $\theta \in (0, 2\pi) \setminus \left\{ \tfrac{\pi}{2}, \pi, \tfrac{3\pi}{2} \right\}$, defines a hypersurface in class~$\mathcal{A}$ with a constant product angle function 
$\Theta=-\cos (2\theta)$ and constant mean curvature given by $n\mathcal{H} = - (k-1)\,\lambda \sin \theta$.
\end{example}

\begin{proposition}\label{thm:cmccfconst}
Let $f \colon M^{n} \to  \mathbb{H}^{k} \times \mathbb{R}^{\,n-k+1}$, 
 $2 \leq k \leq n-1$, 
be a hypersurface in class~$\mathcal{A}$ with constant mean curvature and constant product angle function $\Theta\in (-1, 1)$. 
Then  $f$ is locally as in Example \ref{ex: flat-horospherical}.
\end{proposition}

\begin{proof}
Let $h \colon M_1^{k-1} \to \mathbb{H}^{k}$ and $g \colon M_2^{n-k} \to \mathbb{R}^{n-k+1}$ be 
 isoparametric hypersurfaces, and let $\theta \in (0, 2\pi) \setminus \left\{ \tfrac{\pi}{2}, \pi, \tfrac{3\pi}{2} \right\}$ be such that (locally)
$ M^n = (-\delta, \delta) \times M_1^{k-1} \times M_2^{n-k}$ and 
 $f$ is given by \eqref{definition of f}.

 Assume first that $h$ is a horosphere with principal curvature $\lambda = \pm 1$, so that the mean curvature of the parallel horosphere $h_s$ is also $\lambda$.
Then, \eqref{eq: constant theta 2} becomes 
\[
 (n-k-\nu_g)\,\frac{\mu}{1 - z\mu \tan \theta} =n \mathcal{H}\sec\theta + (k-1)\lambda \tan \theta,
\]
for any $ z \in \mathbb{C} \setminus \{1/(\mu\tan\theta)\}$, where $\mu$ is the unique nonzero principal curvature of $g$ if $\nu_g<n-k$. 
Clearly, the left-hand side of the preceding equation is not constant unless $\nu_g=n-k$, in which case  $f$ is as in Example \ref{ex: flat-horospherical}.

If $h$ is a not a horosphere, then the mean curvature of the parallel hypersurface $h_s$ is 
\[\mathcal{H}^h(s) = \frac{1}{k-1} \Big(m \tanh(\varphi - s) + (k-1-m) \coth(\varphi - s) \Big ), \quad  0 \leq m \leq k-1;\] 
the cases $m=0$ and $m=k-1$ correspond to $h$ being a geodesic hypersphere or an equidistant hypersurface, respectively. 
 Eq.~\eqref{eq: constant theta 2} becomes
\begin{equation}\label{57again}
\begin{array}{l}
 -m \tan\theta\tanh(\varphi - z)
 - (k-1-m) \tan\theta \coth(\varphi - z)
 \vspace{1ex}\\
 \hspace*{10ex} + (n-k-\nu_g)\,\frac{\mu}{1 - z\mu \tan\theta} = n\mathcal{H}\sec \theta,\end{array}
 \end{equation}
for any $z \in \mathbb{C}$ such that 
$$
z \neq \varphi + p\pi i, \quad
z \neq \varphi + \left(\frac{\pi}{2}+p\pi\right)i, \quad
z \neq 1/(\mu\tan\theta), \quad p \in \mathbb{Z}.
$$
Applying \eqref{57again} for  $z=s \in \mathbb{R}\setminus \{\varphi, 1/(\mu\tan \theta)\}$, 
we see that this forces $\varphi = \frac{1}{\mu \tan \theta}$, since otherwise a contradiction arises from the mismatch of singularities.
Substituting this into \eqref{57again} (for  $z=s \in \mathbb{R}$) gives  
\[
 \begin{array}{l}-m \tan\theta\tanh(\varphi - s)
 - (k-1-m) \tan\theta \coth(\varphi - s)\vspace{1ex}\\
\hspace*{10ex} + (n-k-\nu_g)\,\frac{\mu \varphi}{\varphi - s} = n\mathcal{H}\sec \theta,
 \end{array}
\]
for any $s \neq \varphi.$ 
Letting $s \to \infty$ and $s \to -\infty$  yields
$(k-1)\tan \theta = n\mathcal{H}\sec \theta$ and $-(k-1)\tan\theta = n\mathcal{H}\sec \theta$, respectively,
which implies $\tan \theta = 0$, a contradiction.
\end{proof}

\subsubsection{The case $c_1=0=c_2$.} 
We start by giving examples of hypersurfaces of  $\mathbb{R}^{k} \times \mathbb{R}^{n-k+1}=\mathbb{R}^{n+1}$, $2 \leq k \leq n-1$, that are in class  $\mathcal{A}$ and have constant mean curvature and constant product angle function $\Theta\in (-1,1)$. The first is a trivial example.

\begin{example}\label{ex:hyperplanes}
Let $h:\mathbb{R}^{k-1}\to\mathbb{R}^k$ and
$g:\mathbb{R}^{n-k}\to\mathbb{R}^{n-k+1}$ be hyperplanes.
Then, the hypersurface 
 given by \eqref{definition of f} in terms of $h$ and $g$ and any $\theta\in (0,2\pi)\setminus\{\pi/2, \pi, 3\pi/2\}$
 is just an affine hyperplane in $\mathbb{R}^{n+1}$.
\end{example}

\begin{example}\label{ex:cylinderoversphere}
Let
\(
h:\mathbb{S}_{\sigma_1}^m\times\mathbb{R}^{k-m-1}\to\mathbb{R}^k
\)
and
\(
g:\mathbb{S}_{\sigma_2}^p\times\mathbb{R}^{n-k-p}
\to\mathbb{R}^{n-k+1}
\)
be generalized cylinders, where  \(
1\leq m\leq k-1\) and  \(
1\leq p\leq n-k
\).
Write $\sigma_i=1/r_i^2$, $r_i>0$, $i=1,2$, 
and assume that \(
r_2=r_1\sqrt{\frac{p}{m}}
\). Let $\theta =\arctan (r_2/r_1)$, and let $\delta>0$ be such that $h_{s\cos\theta}$ and $g_{s\sin\theta}$ are immersions for all
 $s\in (-\delta,\delta)$. By Corollary \ref{cor: PAC} and \eqref{eq: constant theta}, 
the map \[f\colon  (-\delta,\delta)\times \mathbb{S}_{\sigma_1}^m\times\mathbb{R}^{k-m-1}\times\mathbb{S}_{\sigma_2}^p\times\mathbb{R}^{n-k-p}\to\mathbb{R}^{n+1},\]
given by \eqref{definition of f} with $\theta$ as above,  defines  a minimal hypersurface in class $\mathcal{A}$ with a constant product angle function 
\(\Theta=(p-m)/(p+m)\).
\end{example}

\begin{proposition} \label{eucb}
Let $f:M^n\to\mathbb{R}^k\times\mathbb{R}^{n-k+1}$ be a hypersurface
of class A with constant product angle $\Theta\in(-1,1)$ and
constant mean curvature. Then $f$ is (minimal and) locally given as in Examples~\ref{ex:hyperplanes} or~\ref{ex:cylinderoversphere}.
\end{proposition}

\begin{proof}
By Corollary~\ref{cor: PAC}, there exist isoparametric hypersurfaces
\(
h:M_1^{k-1}\to\mathbb{R}^k\) and \(
g:M_2^{n-k}\to\mathbb{R}^{n-k+1}
\), which are either totally geodesic or generalized cylinders over spheres,
and  $\theta\in (0,2\pi)\setminus\{\pi/2, \pi, 3\pi/2\}$, such that $f$ is locally given by \eqref{definition of f}.

Since the mean curvature of a parallel hypersurface of a generalized cylinder is given by \eqref{euc}, 
which goes to $0$ as $s$ goes to $\infty$, it follows from \eqref{eq: constant theta} that,  in any case, $\mathcal{H}=0$.
Moreover, either $h$ and $g$ are totally geodesic, in which case $f$ is as in Example~\ref{ex:hyperplanes},
or both are generalized cylinders, and hence have
nonzero principal curvatures $\lambda$ and $\mu$ with respective
multiplicities \(
1\leq m\leq k-1\) and \(
1\leq p\leq n-k
\). Since $\mathcal{H}=0$, the complexified version  \eqref{eq: constant theta 2} becomes
\[
\tan\theta\,\frac{m\lambda}
{1-z\lambda}
=\frac{p\mu}
{1-z\mu \tan \theta}.
\]
Comparing the poles and residues gives 
\(
\lambda=\mu\tan\theta\) and
\(
\tan^2\theta={p}/{m}.
\)
Consequently, writing $|\lambda|=1/r_1$ and $|\mu|=1/r_2$, we obtain
\(
r_2=r_1\sqrt{\frac{p}{m}},
\)
with the orientations chosen consistently with
$\lambda=\mu\tan\theta$. Hence $f$ is locally as in
Example~\ref{ex:cylinderoversphere}.
\end{proof}

\begin{rem} 
We have recently become aware of the article \cite{LP}, in which the authors prove that isoparametric hypersurfaces of 
$\mathbb{Q}_{\epsilon_1}^{k} \times \mathbb{Q}_{\epsilon_2}^{n-k+1}$
have necessarily a constant product angle function, an important step towards their classification. Moreover, they obtain such a classification under the additional assumption that the hypersurface has a \emph{distinguished point}, which they prove to imply that all points of the hypersurface are distinguished. The latter condition turns out to be equivalent to the hypersurface belonging to class~$\mathcal{A}$, so Theorem  $2$ in  \cite{LP} classifies the isoparametric hypersurfaces  of $\mathbb{Q}_{\epsilon_1}^{k} \times \mathbb{Q}_{\epsilon_2}^{n-k+1}$ in class 
$\mathcal{A}$.
These correspond to the hypersurfaces in Examples \ref{ex:horospheres} and \ref{ex: flat-horospherical}, besides those that split as products of an isoparametric hypersurface of one of the factors with the identity map in the other.
\end{rem}


\begin{thebibliography}{1}
\bibitem{CT} Carvalho, A. N. S. and  Tojeiro, R., 
\emph{Constant curvature hypersurfaces of cylinders over space forms},
Diff. Geom. Appl. 104 (2026), Paper No. 102415, 12 pp. 


\bibitem{CR} Cecil, T. E., and Ryan, P., \emph{On the Work of Cartan and M\"unzner
on Isoparametric Hypersurfaces}, Axioms 14 (2025), 56.

 



\bibitem{LP} de Lima, R. and Pipoli, G., \emph{Isoparametric hypersurfaces in products of simply connected space forms}, Preprint (available at arXiv 2511.12527).

\bibitem{FKM} Ferus, D., Karcher, and H., M\"unzner, H.F.,  \emph{Cliffordalgebren und neue isoparametrische Hyperfl\"achen}, Math. Z. 177 (1981), 479--502.

\bibitem{LVWY} Haizhong Li, H., Vrancken, L., Wang, X. and Yao, Z., \emph{Hypersurfaces of $\mathbb{S}^2\times \mathbb{S}^2$ with constant sectional curvature},  Calc. Var. Partial Diff. Eq. 63 (2024), no. 7, Paper No. 167, 33 pp.

\bibitem{LVWY2} Haizhong Li, H., Vrancken, L., Wang, X. and Yao, Z., \emph{Hypersurfaces of $\mathbb{H}^2\times \mathbb{H}^2$ with constant sectional curvature}, Preprint (available at arXiv 2606.24686).

\bibitem{KNT} Kim, J., Nikolayevsky, Y. and Tojeiro, R., \emph{Hypersurfaces with constant sectional curvature of a product of two space forms}, Preprint (available at arXiv 2610.15489). 

\bibitem{M}  M\"unzner, H.F.,  \emph{Isoparametric Hyperfl\"achen in Sph\"aren, $I$ and $II$}, Math. Ann. 251 (1980), 57--71 and 256 (1981), 215--232. 

 \bibitem{rs}  Reckziegel, H. and Schaaf, M.,  \emph{De Rham decomposition of netted
manifolds},  Result. Math.  35 (1999), 175--191.

\bibitem{LTV}
Lira, J.H., Tojeiro, R. and Vit\'orio, F.,  \emph{A Bonnet theorem for isometric immersions into products of space forms},
 Arch. Math. 95, 469–479 (2010).
 
\bibitem{MT}
Mendon\c{c}a, B. and Tojeiro, R., \emph{Submanifolds of products of space forms}, Indiana Univ. Math. J. 62 (4) (2013), 1283-1314.

\bibitem{MT2}
Mendon\c{c}a, B. and Tojeiro, R., \emph{Umbilical submanifolds of $\mathbb{S}^n\times \mathbb{R}$},  Canadian J. Math.  66 (2014), 400--428. 

 \bibitem{To} Tojeiro, R., \emph{On a class of hypersurfaces in $\mathbb{S}^n\times \mathbb{R}$ and $\mathbb{H}^n\times \mathbb{R}$},  Bulletin Braz. Math. Soc. 41 (2) (2010), 199-209.

 \bibitem{To2}  Tojeiro, R., \emph{A decomposition theorem for immersions of product manifolds}, Proc. Edinburgh Math. Soc. 59 (2016), 247--269.

\bibitem{U} Urbano, F., \emph{On hypersurfaces of $\mathbb{S}^2\times \mathbb{S}^2$},  Comm. Anal. Geom. 27 (2019), no.~6, 1381–-1416. 






\end{thebibliography}
\end{document}